\documentclass[12pt, a4paper]{amsart}
\usepackage{amsmath,amsthm,amssymb}
\usepackage{amscd}
\usepackage[dvipdfmx]{graphicx}
\usepackage{mathrsfs}
\usepackage[dvipdfmx, usenames]{color}
\usepackage[all]{xy}
\usepackage{ascmac}
\usepackage{bm}
\usepackage{stmaryrd}

\usepackage {blkarray}

\newcounter{numA} 
\newcounter{numB}
\newcounter{numC}
\newtheoremstyle{mystyle}
    {}
    {}
    {\rmfamily}
    {}
    {\rmfamily}
    {.}
    { }
    {\thmname{#1}\thmnumber{ #2}\thmnote{ (#3)}}

\newtheorem{thm}{Theorem}[section]
\newtheorem{lem}[thm]{Lemma}
\newtheorem{cor}[thm]{Corollary}
\newtheorem{prop}[thm]{Proposition}

\newtheorem{conj}[thm]{Conjecture}

\theoremstyle{definition}

\theoremstyle{mystyle}

\newtheorem{rem}[thm]{Remark}        
\newtheorem{example}[thm]{Example}
\newtheorem{defn}[thm]{Definition}

\makeatletter
\renewenvironment{proof}[1][\proofname]{\par
  \pushQED{\qed}%
  \normalfont \topsep6\p@\@plus6\p@\relax
  \trivlist
  \item\relax
  {\rmfamily
  #1\@addpunct{.}}\hspace\labelsep\ignorespaces
}{%
  \popQED\endtrivlist\@endpefalse
}
\makeatother

\numberwithin{equation}{section}

\def\XXint#1#2#3{{\setbox0=\hbox{$#1{#2#3}{\int}$}
\vcenter{\hbox{$#2#3$}}\kern-.5\wd0}}

\newcommand{\R}{\mathbb{R}}

\newcommand{\Lip}{\mathsf{Lip}}

\newcommand{\Cpl}{\mathsf{Cpl}}
\newcommand{\di}{\mathsf{diam}}
\newcommand{\supp}{\mathrm{supp}}

\newcommand{\CDC}{\mathsf{CD}}

\newcommand{\la}{\left\langle}
\newcommand{\ra}{\right\rangle}

\newcommand{\VVert}{{\interleave}} 

\newcommand{\vol}{\mathsf{vol}}
\newcommand{\diam}{\mathsf{diam}}
\newcommand{\KD}{\mathsf{KD}}

\newcommand{\RCD}{\mathsf{RCD}}

\newcommand{\calL}{\mathcal{L}}

\newcommand{\calP}{\mathcal{P}}

\newcommand{\ttb}{\mathtt{b}}

\newcommand{\sfb}{\mathsf{b}}

\newcommand{\RR}{\mathbb{R}} 
\newcommand{\ca}[1]{\mathcal{#1}} 

\newcommand{\ol}[1]{\overline{#1}} 
\newcommand{\ul}[1]{\underline{#1}} 
\newcommand{\wt}[1]{\widetilde{#1}} 

\newcommand{\argmax}{{\mathrm{argmax}}}

\usepackage[abbrev,nobysame]{amsrefs}

\makeatletter
\def\@makefnmark{%
\leavevmode
\raise.9ex\hbox{\check@mathfonts
\fontsize\sf@size\z@\normalfont%
\@thefnmark}%
}
\makeatother

\title[Coarse Ricci Curvature on Hypergraphs]{Coarse Ricci Curvature on Hypergraphs associated with nonlinear Kantorovich difference}
\author[M.Ikeda]{Masahiro Ikeda}
\address[Masahiro Ikeda]{The University of Osaka}
\email{ikeda@ist.osaka-u.ac.jp}
\author[Y.Kitabeppu]{Yu Kitabeppu}
\address[Yu Kitabeppu]{Kumamoto University}
\email{ybeppu@kumamoto-u.ac.jp}
\author[Y.Takai]{Yuuki Takai} 
\address[Yuuki Takai]{Mazda Motor Corporation MAX Project Office}
\email{takai.yu@mazda.co.jp}
\author[T.Uehara]{Takato Uehara} 
\address[Takato Uehara]{Keio University}
\email{takaue@keio.jp}




\thanks{This work was partly supported by JST CREST Grant Number JPMJCR1913, Japan, Grant-in-Aid for Young Scientists Research (No.18K13412, No.19K14581) and Grant-in-Aid for Scientific Research (C) (No.21K11763, No.19K03544), Japan Society for the Promotion of Science, and Grant for Basic Science Research Projects from The Sumitomo Foundation (No.200484).}

\begin{document}

\maketitle
 \begin{abstract}
A hypergraph is a generalization of graphs to be able to represent higher-order relations among entities. Since there has been no canonical notion of random walks on hypergraphs, one cannot naturally extend the notions of coarse Ricci curvature of graphs to hypergraphs. In the present paper, we introduce a new notion of Ricci curvature on hypergraphs associated with a nonlinear Kantorovich difference, which is defined through the resolvent of the nonlinear Laplacian.
We prove that our notion is well-defined regardless of the nonlinearity of the Laplacian via linear programming and gives a generalization of Lin-Lu-Yau's coarse Ricci curvature on graphs. 
Under suitable assumptions of our curvature we obtain a lower bound of nonzero eigenvalues of the Laplacian, a gradient estimate of the heat flow, and a diameter bound of Bonnet-Myers type. 
 \end{abstract}

 %
 %
\section{Introduction}


The Ricci curvature of Riemannian manifolds plays an important role to analyze geometric and analytic properties of the manifolds. In the setting of Riemannian manifolds, though the Ricci tensor needs $C^2$ smooth structure on them, lower bound condition of the Ricci curvature can be described by only the metric and measure. 
More precisely, von Renesse et.al. \cite[Section~1]{von2005transport} proved that
 for any smooth, complete, connected Riemannian manifold $(M,g)$ endowed with the Riemannian distance $d_M$, a volume measure $\vol_g$ on it, the Ricci curvature  
 $\mathrm{Ric}_x(v,v)$ for $x \in M$ and $v \in T_x M$, and any 
 $K\in \RR$,   
the following conditions (1)-(5) are equivalent (\cite[Theorems 1.1 and 1.3]{von2005transport}): 
\begin{enumerate}
 \item (Lower bound of Ricci curvature): 
 $\mathrm{Ric}_x(v,v) \geq K|v|^2$ for any $x \in M$ and $v \in T_x M$.
 
 \item (Convexity of relative entropy): The relative entropy 
 defined in \cite[P.924]{von2005transport}
 is the displacement $K$-convex on the $L^2$-Wasserstein space 
 $(\calP_2(M),W_2)$ defined in \cite[P.923--924]{von2005transport} (see \cite{CEMS:2001}).  
 
 \item (Transportation inequality): For the normalized measure restricted to the ball of radius $r$ centered at $x \in M$
 \begin{align}
  m_{r,x}(A):=\frac{\vol_g(B_r(x)\cap A)}{\vol_g(B_r(x))},\ \ \text{for any Borel $A\subset M$} \notag
 \end{align}
 the following asymptotic estimate holds:
 \begin{align}
  W_1(m_{r,x},m_{r,y})\leq \left(1-\frac{K}{2(n+2)}r^2+o(r^2)\right)d_M(x,y),\ \ \ \text{as}\ r\rightarrow\infty.\notag
 \end{align}
 \item (Contraction property of the gradient flow of entropy): For the gradient flow 
$\Phi:\R_{+}\times\calP_2(M)\rightarrow \calP_2(M)$ with respect to the entropy, 
 \begin{align}
  W_2(\Phi(t,\mu),\Phi(t,\nu))\leq e^{-Kt}W_2(\mu,\nu)\notag
 \end{align}
 holds for any $t\ge 0$ and $\mu,\nu\in\calP_2(M)$.  
 
 \item (Gradient estimate of the heat flow): Let $h_t:L^2(M)\rightarrow L^2(M)$ be the heat flow on $M$. For any $f\in C_c^{\infty}(M)$, $x\in M$, and $t>0$, the following holds:
 \begin{align}
  \vert \nabla h_tf\vert^2(x)\leq e^{-2Kt}h_t\vert \nabla f\vert^2(x).\notag
 \end{align} 
\end{enumerate}
Moreover, the following Bochner inequality (or Bakry-\'{E}mery's curvature-dimension condition) is also equivalent to (1)-(5) (see \cite{AGS:2015BE,AGS:2014Met}): 
\begin{enumerate}
 \item[(6)] (Bochner inequality, curvature-dimension condition of Bakry-\'{E}mery type): Let $\Delta$ be the Laplace-Beltrami operator on $C_c^{\infty}(M)$. For any $f\in C_c^{\infty}(M)$, the following holds: 
 \begin{align}
  \frac{1}{2}\Delta\vert\nabla f\vert^2\geq \la \nabla \Delta f,\nabla f\ra+K\vert \nabla f\vert^2. \notag
 \end{align}
\end{enumerate}
Based on these relations, 
the $\CDC$ (Curvature Dimension) space, which was introduced by Sturm~\cite{Sturm:2006mm1,Sturm:2006mm2} and Lott-Villani~\cite{LV:2007, LV:2009} independently, is defined by using the convexity of entropy on 
the $L^2$-Wasserstein space $(\calP_2(M),W_2)$. In the case of finite dimension, as the entropy, not the relative entropy but the R\'enyi entropy is used. The $\CDC$ space is a metric measure space (not necessarily manifold) whose Ricci curvature is bounded from below in a synthetic sense.
An important point is that the definition of $\CDC$ space is described only in terms 
of measures and metrics. 
For the $\CDC$ space whose dimension is bounded from above, 
many important geometric and functional inequalities such as 
Bishop-Gromov inequality \cite{LV:2009}, Poincar\'{e} inequality \cite{Sturm:2006mm2} and 
Brunn-Minkowski inequality \cite{Rajala:2012poincare} were proved. However the gradient estimate of the heat flow does not hold for generic $\CDC$ spaces (see \cite{OS:2012noncontract}).

After that, the $\RCD$ (Riemannian Curvature-Dimension) space was introduced in \cite{AGS:2014Met, Gigli:2015diffstr}, which is a $\CDC$ space equipped with the infinitesimal Hilbertianity condition (defined by Gigli~\cite[Definition 4.9]{Gigli:2015diffstr}) that its associated Sobolev space $W^{1,2}$ becomes a Hilbert space.
On $\RCD$ space, several theorems such as the $W_2$-contraction of the gradient flow of the relative entropy, the Bochner inequality (Bakry-\'{E}mery's curvature dimension condition) and the gradient estimate of the heat flow have been proved and these are known as equivalent conditions in the setting of manifolds. Many geometric results such as Cheeger-Gromoll's splitting theorem \cite{Gigli:2014splitting},  
Cheng's maximum diameter theorem \cite{Ketterer:2015cone}, isoperimetric inequalities \cite{CM:2017isoperimetric} and so on are also proved and they are known in the setting of Riemannian manifolds.

Both $\CDC$ and $\RCD$ spaces become geodesic metric spaces and $\RCD$ spaces established a position as geodesic spaces whose 
Ricci curvature is bounded from below.

It is quite fundamental how to define a concept of {\it Ricci curvature} on generic metric spaces. 
As we mentioned above, on geodesic metric measure spaces, a synthetic notion of "lower bound of Ricci curvature", called the curvature-dimension condition, is defined. 
On the other hand, there are many different notions of lower bound of Ricci curvature on discrete spaces. In the case of discrete spaces, several definitions whose Ricci curvatures are bounded from below were introduced. 
However there has not been a canonical definition. 
For usual graphs, coarse Ricci curvatures of Ollivier \cite{ollivier2009ricci} and Lin-Lu-Yau \cite{lin2011ricci} are related to the above (3) or (4), the curvature dimension 
condition of Bakry-\'{E}mery type~\cite{schmuckenschlager1999curvature} is related to (6), 
the exponentially curvature-dimension condition is related to the 
Li-Yau inequality \cite{BHLLMY:2015LiYau, munch2017remarks, munch2018li}, and the definitions by Maas \cite{Maas:2011, erbar2012ricci} and by Bonciocat-Sturm \cite{bonciocat2009mass} are related to (2). 
Although all of these definitions stem from the definitions or known facts for geodesic spaces, their relations has not been well understood.



A hypergraph is a generalization of graphs to be able to 
represent relations among not only two but also three or more entities. 
There has been no crucial canonical definition of 
random walks on hypergraphs. Hence one cannot naturally define a notion of curvature on hypergraphs in Olliver's manner \cite[Definition 3]{ollivier2009ricci}. 

In this paper, we introduce a new definition (see Definition \ref{def:lambda-coarse-Ricci-curvature}) of 
a coarse Ricci curvature on hypergraphs, which is well-defined and gives an extension of Lin-Lu-Yau's one on graphs. Our coarse Ricci curvature is defined through a nonlinear Kantorovich difference (Definition \ref{def:NKD}). The Kantrovich difference is inspired by the Kantrovich-Rubinstein duality formula \cite[Theorem 5.10]{Villani:2009oldnew} and defined by the resolvent of the so-called "submodular hypergraph Laplacian" (see (\ref{eq:def-of-Laplacian}) and (\ref{normalizedLap})). The notion of this Laplacian was originally introduced by \cite{louis2015hypergraph, chan2017diffusion, yoshida2019cheeger}. Following \cite{ikeda2018finding}, our Laplacian in this paper is a modification of the definition introduced 
in~\cite{louis2015hypergraph, chan2017diffusion}, and 
a realization of the submodular transformation  
 introduced in \cite{yoshida2019cheeger} when the submodular transformation is a hypergraph (see also Subsection \ref{subsec:submodular-transformations}). 
  
Asoodeh et al \cite{asoodeh2018curvature} introduced a different notion of a Ricci curvature on hypergraphs by using random walks defined by reducing hypergraphs to usual graphs with clique expansion. 

 
 The hypergraph Laplacian was introduced as meaningful from an information engineering point of view, and some research has shown that it can yield good information about hypergraphs. In particular, in \cite{takai2020hypergraph}, it was experimentally proven (in terms of community extraction, especially in terms of spectral graph theory) that hypergraphs can be extracted for their properties as hypergraphs rather than being attributed to ordinary graphs obtained by clique and star expansions. Therefore, we considered that by using this hypergraph Laplacian, the curvature could be defined with more fruitful information about the hypergraph.
However, because this Laplacian is multivalued and nonlinear, there was no canonical way to define the transition probabilities of random walkers using it. For these reasons, we considered the Lin-Lu-Yau definition as a definition using resolvents, and by extending it with resolvents that can be defined even for nonlinear multivalued Laplacians, we thought we could define curvature suitable for hypergraphs.

Recently other notions of Ricci curvature on (directed) hypergraphs were introduced in \cite{EFLSJ:2020,eidi2020ollivier,LRSJ:2018,Akamatsu}.

As connections of the value of our coarse Ricci curvature, 
under similar assumptions of the lower bound of the curvature as Lin-Lu-Yau type, we can deduce a lower bound of nonzero eigenvalues of the normalized Laplacian (Theorem \ref{subsec:eigenvalues}) and a gradient 
estimate of the heat flow of $L^\infty$ type (Theorem \ref{gradient estimate}). Under positive Ricci curuvature, we prove a diameter bound of Bonnet-Myer's type (Theorem \ref{Bonnet-Myers}). It should be noted that these properties do not hold for general $\CDC$ spaces, which implies that one cannot necessarily handle the nonlinearity of our Laplacian.
%

Our arguments for the proofs except for Theorem \ref{thm:existence-of-limit} are applicable to more general settings for submodular transformations \cite[Section 3]{yoshida2019cheeger}, which are vector valued set functions consisting of submodular functions and includes the settings of directed (hyper) graphs and mutual information.

The rest of this paper is organized as follows.
In Section~\ref{sec:preliminaries}, we recall several notions of hypergraphs and basic properties of the submodular hypergraph Laplacian and its resolvent. In Subsection~\ref{subsec:prelim-wasserstein-dist}, we recall basic notions of the metric measure space.
In Subsection~\ref{subsec:review-of-LLY-graph}, we recall the definition of Lin-Lu-Yau's coarse Ricci curvature on usual graphs \cite{lin2011ricci}. In Subsection \ref{rephrase Lin-Lu-Yau's coarse Ricci}, we explain the difficulty of extending Lin-Lu-Yau's coarse Ricci curvature on graphs to hypergraphs and our idea to overcome it. 
In Section \ref{sec:def-of-curvature-hypergraphs}, we introduce the definitions of nonlinear Kantorovich difference and our coarse Ricci curvature on hypergraphs and prove their properties.
In Section~\ref{sec:comparison-on-graph}, we show that in the case of usual graphs, our Ricci curvature is equal to Lin-Lu-Yau's. In Section~\ref{sec:applications}, as connections of our curvature with analytic and geometric properties of hypergraphs, we prove a bound of eigenvalues of the Laplacian, 
a gradient estimate for the heat flow, 
and a Bonnet-Myers type diameter bound. 
In Section ~\ref{sec:examples}, we give several examples of our curvature. 
As the reader seen, strict calculation of curvature for large networks is not easy. However, approximate solutions can be computed as follows. In this paper, the curvature is defined as a limit of the supremum of the differences of values of the resolvent for the hypergraph Laplacian.  The resolvent itself is a personalized PageRank on a hypergraph, as expressed in equation (3) of Section~3.3 in~\cite{takai2020hypergraph}. Using the heat method in~\cite{takai2020hypergraph}, an approximate calculation is possible in a short time.
To calculate an approximated curvature, we also need to calculate a limit of the supremum by running the $1$-Lipschitz function, but this can also be done as a coarse approximation. (It may also be possible to compute the approximation using methods such as design of experiments or Bayesian optimization.)
In Section ~\ref{sec:existenceRicci}, we give a proof of existence of our coarse Ricci curvature for general hypergraphs via linear programming. In Section~\ref{sec:general-settings}, we review 
submodular transformations and the submodular Laplacian 
and give a sufficient condition for a submodular transformation to be able to generalize our curvature notion and theorems to submodular transformations. We show examples of submodular transformations such as directed (hyper)graphs, mutual information etc. in Subsection~\ref{subsec:submodular-examples}.  


  
 %
 %
 \section{Preliminaries}\label{sec:preliminaries}
 
\subsection{Hypergraph} 
\label{subsec:prelim-hypergraphs}
A weighted undirected hypergraph $H = (V, E, w)$ is a triple of a set $V$, 
a set $E\subset V\setminus \{\emptyset\}$ of nonempty subsets of $V$, and a function $w\colon E \to \RR_{>0}$, where $\RR_{>0}:=\{c\in \R\ ;\ c>0\}$. We call an element of $V$ vertex, an element of $E$ hyperedge, and $w$ an edge weight. 
We remark that if $|e| =2$ for any $e \in E$, $H$ is a weighted undirected (usual) graph. Here $|A|$ denotes the cardinality of a set $A$. We say that $H$ is finite if $V$ is finite. 
For $x,y\in V$, we write $x\sim y$ if there exists $e\in E$ such that $x,y\in e$. 
We say that $H$ is connected if for any $x, y\in V$, 
there exists $\{z_i \}_{i=0}^{n}\subset V$ such that $z_0 = x, \ z_{n} = y, \ \text{and} \  \ z_i \sim z_{i+1} \ (i = 0, \dots, n-1)$. Throughout of this paper we assume that 
\begin{equation}
\label{assumptiononhypergraph}
\text{any hypergraph $H$ is finite and 
 connected.} 
\end{equation} 
For $x\in V$, we define the degree of $x$ by $d_x:=\sum_{e\ni x}w_e$. We also define the degree matrix of $V$ by $D:=\mathrm{diag}(d_x)$. 
Since $H$ is connected, then $d_x>0$ for any 
$x\in V$, which implies that $D$ is non-singular, i.e. the inverse $D^{-1}$ of $D$ exists. 
For $S \subseteq V$, the volume of $S$ is defined by $\vol(S):= \sum_{x \in S} d_x$. We introduce a distance $d$ on $V$ defined by 
 \begin{align}
 \label{distance}
  d(x,y):=\min\{n\;;\;\exists\{z_i\}_{i=0}^n,\,z_0=x,\,z_n=y,\;z_i\sim z_{i+1}\},\ \ \ \text{for}\ x,y\in V.
 \end{align}
Then $(V, d)$ becomes a metric space. We define a diameter of $H$, denoted by $\diam(H)$, as that of the metric space $(V,d)$, i.e., $\diam(H) := \max_{x,y\in V} d(x,y)$. We identify the set of all real-valued maps on $V$ with the set $\RR^V$ of vectors indexed by $V$. We denote by $\delta_x\in \R^V$ the characteristic function at $x\in V$, i.e. $\delta_x(z)=1$ if $z=x$ and $\delta=0$ if $z\ne x$.
We define the stationary distribution $\pi \in \RR^V$ by 
$\pi(z) := d_z/\vol(V)$ for $z \in V$.   

\subsection{(Submodular) Laplacian on hypergraph} 
\label{subsec:prelim-laplacian}
We recall the definitions of the submodular hypergraph Laplacian and the normalized version in the sense of Ikeda et al~\cite{ikeda2018finding} and recall their several properties.
We define an inner product $\langle\cdot,\cdot\rangle:\mathbb{R}^V\times\mathbb{R}^V\rightarrow \mathbb{R}$ as 
\[
 \langle f, g \rangle := f^\top D^{-1} g = \sum_{x\in V} f(x)g(x)d_x^{-1}.
\]
Here we use $A^\top$ to denote the transpose of a matrix $A$. We note that $(\RR^V, \langle\cdot,\cdot\rangle )$ is a finite dimensional Hilbert space. We also introduce a norm $\|\cdot\|:\mathbb{R}^V\rightarrow\mathbb{R}_{\ge 0}$ given by $\|f\|:=\langle f,f\rangle^{1/2}$.

We define 
the (submodular) hypergraph 
Laplacian $L\colon \RR^V \to 2^{\RR^V}$ by 
\begin{align}
  & L(f)=Lf:=\left\{\sum_{e\in E}w_e\ttb_e(\ttb_e^{\top}f)\;;\;\ttb_e\in\text{argmax}_{\ttb\in B_e}\ttb^{\top}f\right\},
  \label{eq:def-of-Laplacian}
  \end{align}
where $B_e$ denotes the base polytope for $e\in E$, i.e., the subset of $\RR^V$ defined by  
\begin{align} 
 B_e := \text{Conv}(\{ \delta_x - \delta_y\ ;\  x, y \in e \}).
 \label{eq:base-polytope-hyper}
\end{align}
Here $\text{Conv}(X)$ is the convex hull of 
$X$ in $\RR^V$. This Laplacian $L$ might be multi-valued and nonlinear \cite[Remark 3.2]{Ikeda:2023}, and $L$ is the sub-differential of the convex function $Q\colon \RR^V \to \RR$ defined by  
\[
 Q(f) := \frac{1}{2}\sum_{e\in E} w(e) \max_{x,y\in e} (f(x)-f(y))^2,
\]
(see \cite[Section 2]{yoshida2019cheeger} and \cite[P.15:8]{chan2018spectral}).
Namely the identity $Lf=\partial Q(f)$ holds for any $f\in \R^V$, where $\partial Q:\R^V\rightarrow 2^{\R^V}$ is defined by 
\[
   \partial Q(f):=\{g\in \R^V\ ;\ g^\top(h-f)\le Q(h)-Q(f),\ \text{for any}\ h\in \R^V\}.
\]
 Thus we see that $L$ is a maximal monotone operator (or $-L$ is an $m$-dissipative operator) such that the domain of $L$ is $\R^V$ (see \cite[Lemmas 14, 15]{ikeda2018finding}). 
When the hypergraph $H$ is a usual graph, the (submodular) hypergraph Laplacian $L$ becomes linear and single-valued and $L(f) = \{ (D-A)f \} $, where $A=(w_e)\in \mathbb{R}^{V\times V}$ is the weighted adjacency matrix of the graph (see \cite[Example 3.3]{yoshida2019cheeger} and \cite[Remark 2.3]{Ikeda:2023}).

We also introduce the normalized hypergraph Laplacian 
$\ca{L}\colon \RR^V \to 2^{\RR^V}$ given by
\begin{align}
\label{normalizedLap}
  & \ca{L}(f)= \calL f :=L(D^{-1}f). 
 \end{align}
We note that $\mathcal{L}$ is related to random walk and heat diffusion (see \cite[Subsection 3.3]{takai2020hypergraph}). By \cite[Lemmas 14, 15]{ikeda2018finding}, 
$\ca{L}$ is also a maximal monotone operator on the Hilbert space 
$(\RR^V, \la\cdot,\cdot \ra)$. More strongly, $\calL f$ is the sub-differential of $Q$ at $D^{-1}f$, that is, the identity $\calL f=\partial Q(D^{-1}f)$ holds. We show other properties of $\mathcal{L}$ as follows.
\begin{lem}\label{lem:homogenity-of-Laplacian}
Let $f\in \RR^V$ and $c\in \RR$. Then the following identities hold:
\begin{enumerate}
\item $\ca{L}(cf) = c\ca{L}(f)$,
\item $\calL (f)=\calL (f+c\pi)$, where $\pi \in \RR^V$ is the stationary distribution.
\end{enumerate}
\end{lem}

\begin{proof}
\underline{About (1):}
When $c=0$, the identity is trivial. We note that for any $e\in E$, if $b\in B_e$, then $-b\in B_e$, which implies that $\max_{b\in B_e}\langle \ttb,f\rangle\ge 0$. Assume that $c>0$. Then for any $e\in E$, the identity $\argmax_{\ttb \in B_e} \la \ttb, cf \ra = 
 \argmax_{\ttb \in B_e} \la \ttb, f \ra$ holds, which implies the conclusion. Next we consider the case $c<0$. Then we have 
\[
 \argmax_{\ttb \in {B_e}} \la \ttb, cf \ra = 
 -\argmax_{\ttb \in B_e} \la \ttb, f \ra,
\]
which means the conclusion.

\underline{About (2):} Let $\sfb \in B_e$. Because $\sfb$ is a convex combination of $\delta_x - \delta_y$ for $x, y \in e$ and $D^{-1}\pi(z)=1/\vol(V)$ for $z\in V$, 
we have $\sfb^\top D^{-1} \pi = 0$. Hence we have  
  \begin{align}
   \sfb^{\top}(D^{-1}(f+c\pi))=\sfb^{\top}(D^{-1}f)+c\sfb^{\top}(D^{-1}\pi)=\sfb^{\top}(D^{-1}f), \notag
 \end{align}
which implies that $\calL f=\calL (f+c\pi)$ holds.
\end{proof}

For $\lambda>0$, the resolvent $J_\lambda\colon \RR^V \to 2^{\RR^V}$ of $\mathcal{L}$ is defined as
\begin{equation}
\label{resolvent}
 J_{\lambda}(f)=J_{\lambda}f:=(I+\lambda\calL)^{-1}(f). 
\end{equation}
Here for a multivalued operator $A:\R^V\rightarrow 2^{\R^V}$, the invese $A^{-1}$ is defined by 
\[
 A^{-1}(f) := \{g \in \RR^V ; f \in A(g) \}
\]
with the domain of $A^{-1}$ equal to the range of $A$.
We summarize several properties of $J_{\lambda}$ as follows (see \cite[Corollary 2.10, Lemma~2.11(iii)]{miyadera1992nonlinear} and \cite[Proposition 1.8]{showalter2013monotone}):
\begin{lem}
\label{lem:conti-of-resolvent} 
Let $\lambda>0$ and $J_\lambda$ be defined by (\ref{resolvent}). Then the following holds:
\begin{enumerate}
\item $J_{\lambda}$ is single-valued and its domain and range are $\R^V$. In particular $J_{\lambda}$ is injective.
\item $J_{\lambda}$ is non-expansive, i.e.,  
for any $f, g \in \RR^V$, and any $f'\in J_\lambda(f)$, $g'\in J_\lambda(g)$, the estimate
\[
 \| f' - g' \| \leq  \| f-g \|
\]
holds. Especially, $J_{\lambda}$ is continuous.
\item For any $\mu > 0$ and $f\in \R^V$, the following equation is valid:
\begin{align}
 J_{\lambda}f=J_{\mu}\left(\frac{\mu}{\lambda}f+\frac{\lambda-\mu}{\lambda}J_{\lambda}f\right).
\end{align}
\item (Moreau's theorem): The following identity holds:
\begin{align}
\label{Moreau's theorem}
 J_{\lambda}f=\mathrm{argmin}\left\{\frac{1}{2\lambda}\Vert f-g\Vert^2+Q(D^{-1}g)\;;\;g\in \R^V\right\}.
\end{align} 
\end{enumerate}
\end{lem}
%
%
%
%
%

We derive other properties of $J_\lambda$ from those of the normalized Laplacian $\ca{L}$: 
\begin{lem}\label{lem:homogenity-of-resolvent}
Let $\lambda>0$, $f\in \RR^V$ and $c\in \RR$. Then the following identities hold:
\begin{enumerate}
\item $J_\lambda(cf) = cJ_\lambda(f)$,
\item $J_{\lambda}f=J_{\lambda}(f-c\pi)+c\pi$.
\end{enumerate}
\end{lem}

\begin{proof}
%
\underline{About (1):} When $c=0$, the identity is trivial. For nonzero $c\in \RR$, let $g:= J_\lambda(cf)$. Then, 
$cf \in (I+\lambda\ca{L})(g)$ holds by the definition of $J_\lambda$. 
Thus, we have $f \in (I + \lambda\ca{L})(c^{-1}g)$ by Lemma \ref{lem:homogenity-of-Laplacian}. This implies $c^{-1}g = J_\lambda(f)$, hence $g = cJ_\lambda(f)$. 

\underline{About (2):} We set $g:=J_{\lambda}f$ and $h:=J_{\lambda}(f-c\pi)$. Then there exist $g'\in\calL g$ and $h'\in\calL h$ such that the identities $f=g+\lambda g'$ and $f-c\pi=h+\lambda h'$ hold. Thus, we have  
   \begin{align*}
   (I+\lambda\ca{L})(g) &\ni g+\lambda g' =f=h+\lambda h'+c\pi \\ 
   &\in (h+c\pi) +\lambda \ca{L}(h)\notag = (h+c\pi)+\lambda \ca{L}(h + c \pi) = (I+\lambda\ca{L})(h + c\pi).\notag
   \end{align*}
Here the inclusion follows from Lemma \ref{lem:homogenity-of-Laplacian}. Therefore, acting $J_{\lambda}=(I+\lambda \ca{L})^{-1}$ to the both sides, we get $g=h+c\pi$ 
because $J_\lambda$ is injective. 

%
%
\end{proof}

Since $\mathcal{L}$ is a maximal monotone operator such that the domain is $\mathbb{R}^V$, by \cite[Theorem 4.2]{miyadera1992nonlinear}, the heat semigroup
$\{e^{-t\calL}\}_{t\ge 0}$ is well defined on $\mathbb{R}^V$ and the following identity holds:
\begin{align}
   e^{-t\mathcal{L}} f=\lim_{\lambda\downarrow0}J_{\lambda}^{[t/\lambda]}f,\ \ \ t\ge 0\ \text{and}\ f\in \R^V.
  \label{eq:heat-and-resolvent} 	
\end{align}
Here $[a]$ is the maximum integer less than or equal to $a\in \RR$. For $f\in \mathbb{R}^V$, we set
 \begin{align}
  \VVert \calL f\VVert:=\inf\left\{ \Vert f'\Vert\;;\;f'\in\calL f\right\}\notag.
 \end{align}
Then by \cite[Lemma~2.11 (ii)]{miyadera1992nonlinear} the following holds:
 \begin{align}
  \Vert J_{\lambda}f-f\Vert \leq \lambda{\VVert} \calL f\VVert,\ \ \ \text{for}\ \lambda>0\ \text{and}\ f\in \R^V.\label{eq:bound-by-operator-norm}
 \end{align}
Since $\calL f$ is a closed convex set 
by \cite[Lemma~2.15]{miyadera1992nonlinear}, 
there exists a unique $f'\in\calL f$ such that $\Vert f'\Vert=\VVert \calL f\VVert$ by \cite[Lemma 2.19]{miyadera1992nonlinear}. 
We set $\calL^0 f$ as this $f'$, i.e. $\calL^0 f=f'$. This defines a single-valued operator 
$\calL^0\colon \RR^V \to \RR^V$, called the 
{\em canonical restriction} of $\calL$. 
Then by \cite[Lemma~2.22 and Theorem~3.5]{miyadera1992nonlinear}, the following identities hold: 
 \begin{align}
  -\calL^0f=\lim_{\lambda\downarrow 0}\lambda^{-1}(J_{\lambda}f-f)=\lim_{t\downarrow0}t^{-1}(e^{-t\mathcal{L}}f-f).\label{eq:L0-and-limit-of-resolvent}
 \end{align}
\subsection{$L^1$-Wasserstein distance}
\label{subsec:prelim-wasserstein-dist}
Let $(X,\mathsf{d},m)$ be a metric measure space, that is, $(X,\mathsf{d})$ is a complete separable metric space and $m$ is a locally finite Borel measure on $X$. 
We set $\calP(X)$ as the set of all Borel probability measures. For $\mu,\nu\in\calP(X)$, a measure $\xi\in\calP(X\times X)$ is called a coupling between $\mu$ and $\nu$ if 
\[  \xi(A\times X)=\mu(A),\ \ \ 
  \xi(X\times A)=\nu(A)
\]holds for any Borel set $A\subset X$. We set $\Cpl(\mu,\nu)$ as the set of all couplings between $\mu$ and $\nu$. Since the product measure $\mu\otimes \nu$ of $\mu$ and $\nu$ is a coupling between $\mu$ and $\nu$, $\Cpl(\mu,\nu)$ is nonempty. We define the $L^1$-Wasserstein space $\calP_1(X)$ by 
\begin{align}
 \calP_1(X):=\left\{\mu\in\calP(X)\;;\;\int_X\mathsf{d}(x,o)\,\mu(dx)<\infty\;\text{for a point }o\in X\right\}.\notag
\end{align} 
  For $\mu,\nu\in\calP_1(X)$, the $L^1$-Wasserstein distance between them, denoted by $W_1(\mu,\nu)$, is defined as 
  \begin{align}
\label{L1-Wasserstein}
   W_1(\mu,\nu):=\inf\left\{\int_{X\times X}\mathsf{d}(x,y)\,\xi(dx,dy)\;;\;\xi\in\Cpl(\mu,\nu)\right\}.
  \end{align}
It is known that $W_1$ is a metric on $\calP_1(X)$ and the following duality formula for $W_1$ holds (see \cite[Theorem 5.10]{Villani:2009oldnew} for example).
 \begin{prop}[Kantorovich-Rubinstein duality]\label{prop:KR-duality}
  For $\mu,\nu\in\calP_1(X)$, 
  \begin{align}
   W_1(\mu,\nu)=\sup\left\{\int_Xf\,d\mu-\int_Xf\,d\nu\;;\;f\text{is 1-Lipschitz}\right\}\label{eq:KDlin}
  \end{align}
  holds. Here we say that $f$ is $1$-Lipshitz if for any $x,y\in X$, the estimate $|f(x)-f(y)|\le \mathsf{d}(x,y)$ holds.
 \end{prop}
 We call a $1$-Lipschitz function $f$ that realizes the supremum of (\ref{eq:KDlin}) a {\em Kantorovich potential}.

\subsection{Coarse Ricci curvature on usual graphs of Lin-Lu-Yau type}
\label{subsec:review-of-LLY-graph}
In this subsection we recall the definition of the coarse Ricci curvature on usual graphs of Lin-Lu-Yau's type \cite[P609]{lin2011ricci}. As shown in Proposition \ref{prop:campare-LLY} below, our definition of the curvature on hypergraphs gives a generalization of the Lin-Lu-Yau's type.
Let $G=(V,E)$ be a simple graph, that is, $V$ is a set and $E\subset V\times V\setminus \{(x,x)\;;\;x\in V\}$. Here we do not distinguish $\{x,y\}$ and $\{y,x\}\in E$. For $x,y \in V$, $x\sim y$ means $\{x,y\}\in E$. Given $x,y\in V$, a sequence of points $\{z_i\}_{i=0}^n$ is called a path from $x$ to $y$ if $z_0=x$, $z_n=y$, $z_i\sim z_{i+1}$ for $i=0,\cdots, n-1$, and $n$ is called the length of path. 
The distance $d(x,y)$ of $x,y\in V$ as the least number of lengths of paths from $x$ to $y$. A path $\{z_i\}_{i=0}^n$ is said to be geodesic if it realizes the distance between $z_0$ and $z_n$. We introduce a weight function $w:V\times V\rightarrow \R_{\geq 0}$ such that $w(x,y)>0$ if and only if $x\sim y$. The  degree of $x\in V$ is defined by $d_x:=\sum_{y\in V}w(x,y)$. Now that $G$ is a usual graph, the normalized Laplacian $\mathcal{L}$ defined by (\ref{normalizedLap}) becomes linear and single-valued and $\mathcal{L}=\{I-AD^{-1}\}$, where $A:= (w(x,y))_{x,y}\in \RR^{V\times V}$ is the adjacency matrix of $G$. 
\smallskip
\par For $\alpha\in (0,1)$ and $x \in V$, we introduce a function $m_x^{\alpha}$ from $V$ to $\mathbb{R}_{>0}$ defined by 
\begin{align}
 m_x^{\alpha}(y):=
 \begin{cases}
  \alpha&\text{if }y=x,\\
  \frac{1-\alpha}{d_x}w(x,y)&\text{if }y\sim x,\\
  0&\text{otherwise}.  
 \end{cases}\notag
\end{align}
We can regard $m_x^{\alpha}$ as a probability measure on $V$ and $m_x^{\alpha}\in \calP_1(V)$. For $\alpha\in (0,1)$ and two distinct vertices $x$ and $y$, we define the $\alpha$-lazy coarse Ricci curvature $\kappa^{\alpha}(x,y)$ between $x$ and $y$ by 
\begin{align}
\label{alpha-lazy}
 \kappa^{\alpha}(x,y):=1-\frac{W_1(m_x^{\alpha},m_y^{\alpha})}{d(x,y)}.
\end{align}
Lin-Lu-Yau \cite{lin2011ricci} introduced the coarse Ricci curvature $\kappa^{\mathrm{LLY}}(x,y)$ on $G$ given by 
\begin{align}
 \kappa^{\mathrm{LLY}}(x,y):=\lim_{\alpha\uparrow1}\frac{\kappa^{\alpha}(x,y)}{1-\alpha}\label{eq:LLYdef}
\end{align}
and proved its several properties including existence of the limit (\ref{eq:LLYdef}). 

\subsection{Rephrase Lin-Lu-Yau's coarse Ricci curvature}
\label{rephrase Lin-Lu-Yau's coarse Ricci}
We note that for $\alpha\in (0,1)$ and $x\in V$, $m_x^{\alpha}$ can be written as 
\[
 m_x^\alpha = (\alpha I + (1-\alpha)AD^{-1})\delta_x 
 = (I - (1-\alpha)\ca{L})\delta_x.
\]
However it is difficult to generalize $W_1(m_x^\alpha, m_y^\alpha)$ to the case of hypergraphs since our normalized hypergraph Laplacian $\mathcal{L}$ given by (\ref{normalizedLap}) is generally multi-valued and nonlinear. To overcome the difficulty, we give the following observation.

By using the above expression, the identities hold: 
\begin{align*}
 \int_V f \, d m_x^\alpha &
 = f^\top m_x^\alpha 
 = f^\top (I - (1-\alpha)\ca{L})\delta_x
 = \la (I - (1-\alpha)\ca{L})Df, \delta_x\ra.  
\end{align*}
Let $y\in V$. By the Kantorovich-Rubinstein duality 
(Proposition~\ref{prop:KR-duality}), 
$W_1(m_x^\alpha, m_y^\alpha)$ can be written as  
\begin{align}
 W_1(m_x^{\alpha},m_y^{\alpha}) = 
 \sup\left\{\la (I - (1-\alpha)\ca{L})Df, \delta_x-\delta_y \ra \;;\;f\text{ is a } 1\text{-Lipschitz}\right\}. 
 \label{KR-duality-graph}
\end{align}
Let $\lambda:= 1-\alpha\in (0,1)$ and $J_\lambda:= (I +\lambda \ca{L})^{-1}$ be the resolvent of $\mathcal{L}$. Then for any $f\in \R^V$, the identity 
\begin{align}
 (I-\lambda\ca{L})(g) = J_\lambda(g) + O(\lambda^2) 
 \label{eq:infiniticimal-estimate-resolvent-linear}
\end{align}
holds for sufficiently small $\lambda>0$.
Indeed, since $G$ is a usual graph, $\ca{L}$ is a matrix, which enables us to apply Neumann 
series expansion to get
\begin{align*}
    J_{\lambda}g=g-\lambda\calL g+\sum_{k=2}^{\infty}(-\lambda\calL)^k g 
    = (I -\lambda \ca{L})g + O(\lambda^2),\ \ \ \text{as}\ \lambda\rightarrow +0.
\end{align*}
We introduce a $\lambda$-linear Kantorovich difference 
$\KD_\lambda(x,y)$ as 
\begin{align}
 \KD_\lambda(x,y) := 
 \sup\left\{\la J_\lambda Df, \delta_x-\delta_y \ra \;;\;f\text{ is } 1\text{-Lipschitz}\right\}. \label{eq:KD-graph}
\end{align}
Then by the estimate~\eqref{eq:infiniticimal-estimate-resolvent-linear}, we can show 
\begin{align}
 W_1(m_x^\alpha, m_y^\alpha) = \KD_\lambda(x,y) + o(\lambda),\ \ \ \text{as}\ \lambda\rightarrow +0.
 \label{estimate-W1-KD-linear} 
\end{align}
We will prove this identity rigorously in Section~\ref{sec:comparison-on-graph}. 
The crucial point to extend the definition of Lin-Lu-Yau's curvature notion on graphs
to hypergraphs is that $\lambda$-Kantrovich difference $\KD_\lambda(x,y)$ can be extended naturally to hypergraphs, since the resolvent $J_\lambda$ of our hypergraph Laplacian is single-valued (Lemma \ref{lem:conti-of-resolvent}).


%
%
\section{Definition of coarse Ricci curvature on hypergraphs}
\label{sec:def-of-curvature-hypergraphs}

Let $H=(V,E,w)$ be a weighted undirected hypergraph.

\subsection{Nonlinear Kantorovich difference} 
\label{subsec:def-of-nonlinear-KD}
In this subsection we introduce a notion of nonlinear Kantrovich difference, which is a natural generalization of 
\eqref{eq:KD-graph}, and prove its several fundamental properties. They are used to derive several properties of our coarse Ricci curvature on hypergraphs (see Subsection \ref{subsec:def-of-coarse-ricci-curvature-hypergraph}).

Let $d:V\times V\rightarrow \R_{\ge 0}$ be a distance defined by (\ref{distance}) and $K>0$. A function $f:V\rightarrow \mathbb{R}$ is said to be \emph{weighted $K$-Lipschitz} 
 if $D^{-1}f$ is a $K$-Lipschitz function with respect to $d$, that is, $f$ satisfies 
\[
 \frac{f(x)}{d_x} - \frac{f(y)}{d_y} \leq K d(x,y)
\]
for any $x, y \in V$. The left hand side can be written as $\la f,\delta_x-\delta_y\ra$. We denote the set of all weighted $K$-Lipschitz functions on $V$ as 
  $\Lip_w^K(V)$. Note that if $f\in \Lip_w^K(V)$, then so is $-f$.
  \begin{defn}[$\lambda$-nonlinear Kantorovich difference]
  \label{def:NKD}
Let $\lambda>0$, $J_{\lambda}$ be the resolvent (\ref{resolvent}) of the normalized hypergraph Laplacian $\mathcal{L}$ and let $x,y\in V$. Then the \emph{$\lambda$-nonlinear Kantorovich difference} $\KD_{\lambda}(x,y)$ of $x$ and $y$ is defined by
   \begin{align}
    \KD_{\lambda}(x,y):=\sup\left\{\langle J_{\lambda}f,\delta_x\rangle-\langle J_{\lambda}f,\delta_y\rangle\;;\; f\in\Lip_w^1(V)\right\}.\notag
   \end{align}
  \end{defn}
 \begin{rem}
 \label{basicpropertiesKD}
\begin{enumerate} 
\item Since $0\in \Lip_w^1(V)$, the estimate $\KD_{\lambda}(x,y)\ge 0$ holds.
\item Let $\lambda>0$, $x,y\in V$ and $f\in \Lip_w^1(V)$. We write formally 
  \begin{align}
   \int f\,d\mu^{\lambda}_x:=\langle J_{\lambda}f,\delta_x\rangle\notag.
  \end{align}
If the hypergraph $H$ is a usual graph, $\mu^{\lambda}_x$ becomes a measure. 
 \end{enumerate}  
 \end{rem}

We introduce a weighted maximum norm $\|\cdot\|_{\infty}$ on $\R^V$ given by
\begin{align}
    \Vert f\Vert_{\infty}:=\max_{x\in V} \left|\frac{ f(x)}{d_x} \right|.\notag
   \end{align}
We also introduce a bounded and closed subset of $\Lip_w^1(V)$ given by
\[
 \wt{\Lip_w^1}(V):= \{ f\in \Lip_w^1(V)\ ;\ \| f\|_\infty \leq \diam(H) \}.
\]

\begin{rem}
The subset $\wt{\Lip_w^1}(V)$ is a compact subset of $(\RR^V, \la\cdot,\cdot \ra)$ since $\R^V$ is of finite dimensional.
\end{rem}

We can restrict the class of functions i.e. $\Lip_w^1(V)$ in the definition of $\KD_{\lambda}$ to the compact subset $\wt{\Lip_w^1}(V)$: 
  \begin{prop}\label{prop:finitenorm}
%
Let $\lambda>0$ and $x,y\in V$. Then the following identity holds:
   \begin{align}
    \KD_{\lambda}(x,y)=\sup\left\{\la J_{\lambda}f,\delta_x-\delta_y\ra\;;\;f\in\wt{\Lip_w^1}(V) \right\}.\notag
   \end{align} 
  \end{prop}
 \begin{proof}
Let $f\in \Lip_w^1(V)$. Take $y_0\in V$ such that $\la f,\delta_{y_0}\ra=\min_{z\in V}\la f,\delta_z\ra$ and we set $\theta:=\la f,\delta_{y_0}\ra$ and $F:=f-\theta\cdot \vol(V) \pi$. Then for any $x,y\in V$, the following identities hold: 
  \begin{align}
   \la F,\delta_x-\delta_y\ra=\frac{f(x)-\theta d_x}{d_x}-\frac{f(y)-\theta d_y}{d_y}=\la f,\delta_x-\delta_y\ra,\notag
  \end{align}
which implies $F\in \Lip_w^1(V)$. Thus by the identity $\la F, \delta_{y_0}\ra=0$, for any $x\in V$, the following estimates hold:
  \begin{align}
   \la F,\delta_x\ra=\la F,\delta_x\ra-\la F,\delta_{y_0}\ra\leq d(x,y_0)\leq \diam(H).\notag
  \end{align}
This means that $\Vert F\Vert_{\infty}\leq \diam(H)$.  By Lemma~\ref{lem:homogenity-of-resolvent}, the identity $J_{\lambda}F=J_{\lambda}f-\theta\cdot \vol(V)\pi$ holds. Thus for any $x,y\in V$, the identities hold 
  \begin{align}
   \frac{J_{\lambda}F(x)}{d_x}-\frac{J_{\lambda}F(y)}{d_y}=\frac{J_{\lambda}f(x)-\theta d_x}{d_x}-\frac{J_{\lambda}f(y)-\theta d_y}{d_y}=\frac{J_{\lambda}f(x)}{d_x}-\frac{J_{\lambda}f(y)}{d_y},\notag
  \end{align}
which implies the desired property. 



 \end{proof}

We prove finiteness of $\KD_\lambda(x,y)$ and an upper bound of $\KD_\lambda$. The following lemma implies existence of the lower coarse Ricci curvature (see Remark \ref{lowerboundlowercoarse}).
 \begin{lem}\label{lem:KDfinite}
Let $\lambda>0$ and $x,y\in V$. Then the following estimate hold:
   \begin{equation}
    \KD_{\lambda}(x,y)\leq 2\lambda \diam(H)\vol(V)^{1/2}\max_{z\in V}d_z^{-1/2}+d(x,y). \label{eq:upper-bound-of-KD} \qedhere 
    \notag 
   \end{equation}
Moreover, the following inequality holds:
\[
   \max_{x,y\in V}\KD_{\lambda}(x,y) \le (2\lambda \vol(V)^{1/2}\max_{z\in V}d_z^{-1/2}+1)\diam(H)<\infty.
\]
 \end{lem}
  \begin{proof}


We first show that for any $f\in \Lip_w^1(V)$, the following estimate holds:
\begin{equation}
\label{estimateLf}
\VVert \calL f\VVert\leq \diam(H)\vol(V)^{1/2}.
\end{equation} Let $e\in E$ and $\ttb_e\in\mathrm{argmax}_{\ttb\in B_e}\ttb^{\top}(D^{-1}f)$. Since $f\in \Lip_w^1(V)$, the following estimates hold:
\[
\ttb_e^{\top}(D^{-1}f)=\max_{x,y\in e}\left\vert f(x)/d_x-f(y)/d_y\right\vert
\le \max_{x,y\in e}d(x,y)\leq \diam(H).
\]
We set $f':=\sum_{e}w_e\ttb_e(\ttb_e^{\top}(D^{-1}f))\in\calL f$. 
We note that $\vert \ttb_e(x)\vert\leq 1$ for any $x\in V$ since $\ttb_e\in B_e$. Then  
\begin{align}
 \vert f'(x)\vert=\left\vert \sum_{e\in x}w_e\ttb_e^{\top}(D^{-1}f)\ttb_e(x)\right\vert\leq \diam(H)\sum_{e\in x}w_e = \diam(H)d_x.\label{eq:upperbound-of-L(f)}
\end{align}
Consequently, we obtain 
\begin{align}
 \VVert \calL f\VVert^2\leq \la f',f'\ra=\sum_{x\in V}f'(x)^2d_x^{-1}\leq \diam(H)^2\sum_{x\in V}d_x=\diam(H)^2\vol(V).\notag
\end{align}
   \smallskip
Next we go back to the proof. For any $f\in \Lip_w^1(V)$, by (\ref{eq:bound-by-operator-norm}), the estimates hold:
   \begin{align}
    &\langle J_{\lambda}f,\delta_x\rangle-\langle J_{\lambda}f,\delta_y\rangle=\langle J_{\lambda}f-f,\delta_x\rangle+\langle f,\delta_x\rangle-\langle f,\delta_y\rangle-\langle J_{\lambda}f-f,\delta_y\rangle\notag\\
    &\leq \Vert J_{\lambda}f-f\Vert(\Vert \delta_x\Vert+\Vert \delta_y\Vert)+d(x,y)\notag\\
    &\leq \lambda\VVert \calL f\VVert(d_x^{-1/2}+d_y^{-1/2})+d(x,y)\notag\\
    &\leq 2\lambda \diam(H)\vol(V)^{1/2}\max_{z\in V}d_z^{-1/2}+d(x,y).\notag
   \end{align}
Because the last quantity is independent of $f$, we take the supremum with respect to $f$ to get the conclusion of this lemma.
  \end{proof}

Next we prove that for any $\lambda>0$, $\KD_{\lambda}(\cdot,\cdot)$ is a distance function on $V$:

  \begin{prop}\label{prop:distance}
Let $\lambda>0$ and $x,y\in V$. Then the following holds: 
   \begin{enumerate}
    \item $\KD_{\lambda}(x,y)=0$ if and only if $x=y$.
    \item $\KD_{\lambda}(x,y)=\KD_{\lambda}(y,x)$.
    \item For $z\in V$, the triangle inequality $\KD_{\lambda}(x,z)\leq \KD_{\lambda}(x,y)+\KD_{\lambda}(y,z)$ holds.
   \end{enumerate}
  \end{prop}
  \begin{proof}
"If\ " part of (1) and (2) follow from the definition. 
We prove "only if\ " part of (1). 
We assume that $\KD_{\lambda}(x,y)=0$. 
Then for any $f\in \Lip_w^1(V)$, $\langle J_\lambda f, \delta_x - \delta_y\rangle = 0$. We can see that the identity 
$\{ c f ; f \in \Lip_w^1(V), c \in \RR \} = \RR^V 
$ holds. Indeed, let $h\in \R^V\backslash\{0\}$. Set $c:=2\|h\|_{\infty}>0$ and $h=cf$. Then for any $u,v\in V$ with $u\ne v$, noting that $d(u,v)\ge 1$, the following estimates hold:
\begin{align*}
   \left|\frac{f(u)}{d_u}-\frac{f(v)}{d_v}\right|=\frac{1}{c}\left|\frac{h(u)}{d_u}-\frac{h(v)}{d_v}\right|\le d(u,v).
\end{align*}
Thus by Lemma~\ref{lem:conti-of-resolvent}, the identity $\langle g, \delta_x - \delta_y\rangle = 0$ holds for any $g \in \RR^V$. 
The non-degeneracy of the inner product implies $\delta_x=\delta_y$, which means $x=y$. 
  Next we prove (3). For any $\epsilon>0$, there exists $f=f_{\varepsilon}\in \Lip_w^1(V)$ such that 
$\KD_{\lambda}(x,z)\leq \langle J_{\lambda}f,\delta_x\rangle-\langle J_{\lambda}f,\delta_z\rangle+\epsilon$.
Thus, we have 
   \begin{align}
    \KD_{\lambda}(x,z)&\leq 
    \langle J_{\lambda}f,\delta_x\rangle-\langle J_{\lambda}f,\delta_y\rangle+\langle J_{\lambda}f,\delta_y\rangle-\langle J_{\lambda}f,\delta_z\rangle
        +\epsilon\notag\\
    &\leq \KD_{\lambda}(x,y)+\KD_{\lambda}(y,z)+\epsilon.\notag
   \end{align}
Since $\epsilon>0$ is any positive number, the conclusion holds. 
  \end{proof}
 
Next we study how the function $\KD_{\lambda}$ changes with respect to $\lambda$. We can prove the following Lipshitz continuity: 
\begin{prop}
Let $\lambda,\mu>0$ Then the following estimate holds:
\begin{equation}
\label{LipshitzKD}
\sup_{x,y\in V}|\KD_{\lambda}(x,y)-\KD_{\mu}(x,y)|\leq 2\diam(H)\vol(V)^{1/2}\max_{z\in V}d_z^{-1/2}\vert \lambda-\mu\vert.
\end{equation}
\end{prop}
\begin{proof}
Let $x,y\in V$. Let $g\in \Lip_w^1(V)$. By (2) and (3) of Lemma \ref{lem:conti-of-resolvent} and (\ref{estimateLf}), the following estimates hold: 
\begin{align}
 &\la J_{\lambda}g-J_{\mu}g,\delta_x-\delta_y\ra=\la J_{\mu}\left(\frac{\mu}{\lambda}g+\frac{\lambda-\mu}{\lambda}J_{\lambda}g\right)-J_{\mu}g,\delta_x-\delta_y\ra\notag\\
 &\leq \left\Vert J_{\mu}\left(\frac{\mu}{\lambda}g+\frac{\lambda-\mu}{\lambda}J_{\lambda}g\right)-J_{\mu}g\right\Vert\cdot\Vert \delta_x-\delta_y\Vert\notag\\
 &\leq \left\Vert \frac{\mu}{\lambda}g+\frac{\lambda-\mu}{\lambda}J_{\lambda}g-g\right\Vert\cdot 2\max_{z\in V}d_z^{-1/2}\notag\\
 &\leq 2\max_{z\in V}d_z^{-1/2}\vert \lambda-\mu\vert \Vert \calL^0f\Vert \leq 2\diam(H)\vol(V)^{1/2}\max_{z\in V}d_z^{-1/2}
\vert \lambda-\mu\vert.\notag
\end{align}
Let $\epsilon>0$. Then there exists $f\in \Lip_w^1(V)$ such that $\KD_{\lambda}(x,y)-\epsilon\leq \la J_{\lambda}f,\delta_x-\delta_y\ra$. Thus by the above estimates, the following inequalities hold:
\begin{align}
 \KD_{\lambda}(x,y)-\epsilon&\leq \la J_{\lambda}f,\delta_x-\delta_y\ra=\la J_{\lambda}f-J_{\mu}f,\delta_x-\delta_y\ra+\la J_{\mu}f,\delta_x-\delta_y\ra\notag\\
 &\leq 2\diam(H)\vol(V)^{1/2}\max_{z\in V}d_z^{-1/2}\vert \lambda-\mu\vert+\KD_{\mu}(x,y).\notag
\end{align}
Since $\epsilon>0$ is arbitrary, we obtain $\KD_{\lambda}(x,y)-\KD_{\mu}(x,y)\leq \text{RHS. of (\ref{LipshitzKD})}$. Changing the role of $\mu$ and $\lambda$, we have the conclusion.  
  \end{proof}

Next we prove that for $\lambda>0$ and $x,y\in V$, there exists a function in $\Lip_w^1(V)$ which 
attains $\KD_\lambda(x,y)$. We call such function a \emph{$\lambda$-nonlinear Kantorovich potential}.
  \begin{prop}\label{prop:nKP}
Let $\lambda>0$ and $x,y\in V$. Then there exists $f\in \Lip_w^1(V)$ such that the identity $\la J_{\lambda}f,\delta_x\ra-\la J_{\lambda}f,\delta_y\ra=\KD_{\lambda}(x,y)$ holds. Namely the following identity holds:
\[
    \KD_{\lambda}(x,y)=\max\left\{\langle J_{\lambda}f,\delta_x\rangle-\langle J_{\lambda}f,\delta_y\rangle\;;\; f\in\Lip_w^1(V)\right\}.
\]
  \end{prop}
 \begin{proof}
  Let $\{f_n\}\subset \Lip_w^1(V)$ be a maximizing sequence of 
 $\KD_{\lambda}(x,y)$. As mentioned in Proposition \ref{prop:finitenorm}, without loss of generality, we may assume $\sup_n\Vert f_n\Vert_{\infty}\leq \di(H)$. Since $\wt{\Lip_w^1}(V)$ is a compact subset of the finite dimensional Euclidean space 
 $(\RR^V, \la\cdot,\cdot\ra)$ by Proposition~\ref{prop:finitenorm}, 
 thus a sequentially compact subset. 
Hence $\{f_n\}$ has a subsequence 
$\{f_{n_j}\}_j$ which converges to an element $f'$ in $\Lip_w^1(V)$. 
Since $J_{\lambda}$ is continuous by Lemma~\ref{lem:conti-of-resolvent}, 
we have $J_{\lambda}(f_{n_j})\rightarrow J_{\lambda}(f')$ as $j\rightarrow\infty$, which implies 
 \begin{align}
  \KD_{\lambda}(x,y)=\lim_{j\rightarrow \infty}\la J_{\lambda}(f_{n_j}),\delta_x-\delta_y\ra=\la J_{\lambda}(f'),\delta_x-\delta_y\ra.\notag
 \end{align}
 \end{proof}
%

\begin{cor}
Let $\lambda>0$. Then the following estimate holds:
\[
  \sup_{x,y\in V}|\KD_{\lambda}(x,y)-d(x,y)|\leq 2\lambda \diam(H)\vol(V)^{1/2}\max_{z\in V}d_z^{-1/2}.
\]
\end{cor}

\begin{proof}
Let $x,y\in V$. By Proposition \ref{prop:nKP}, there exists $f\in \Lip_w^1(V)$ with $\langle f,\delta_x-\delta_y\rangle\leq d(x,y)$ which attains $\KD_{\lambda}(x,y)$. In the similar manner as above, the following estimates hold
\begin{align}
    &\KD_{\lambda}(x,y)\geq\langle J_{\lambda}f-f,\delta_x\rangle+\langle f,\delta_x\rangle-\langle f,\delta_y\rangle-\langle J_{\lambda}f-f,\delta_y\rangle\notag\\
    &\geq -\lambda\VVert \calL f\VVert(d_x^{-1/2}+d_y^{-1/2})+d(x,y)\notag\\
    &\geq -2\diam(H)\lambda \vol(V)^{1/2}\max_{z\in V}d_z^{-1/2}+d(x,y).\notag
   \end{align}
By combining this and Lemma \ref{lem:KDfinite}, the conclusion holds.
\end{proof}

\subsection{Coarse Ricci curvature on hypergraphs} 
\label{subsec:def-of-coarse-ricci-curvature-hypergraph}
Let $H=(V,E,w)$ be a weighted undirected hypergraph and $x,y\in V$. In this subsection, we introduce a coarse Ricci curvature on $H$ along with $x,y$, denoted by $\kappa(x,y)$, and show its fundamental properties.
\begin{defn}[Coarse Ricci curvature on hypergraphs]\label{def:lambda-coarse-Ricci-curvature}
Let $\lambda>0$, $x$ and $y$ be two distinct vertices and $\KD_{\lambda}(x,y)$ be the $\lambda$-nonlinear Kantorovich difference defined in Definition \ref{def:NKD}. Then the {\em $\lambda$-coarse Ricci curvature} along with $x, y$, denoted by $\kappa_{\lambda}(x,y)$, is defined by
 \begin{align}
 \label{lambda-coarse}
  \kappa_{\lambda}(x,y):=1-\frac{\KD_{\lambda}(x,y)}{d(x,y)}.
 \end{align}
The {\em lower coarse Ricci curvature} $\underline\kappa(x,y)$ and the {\em upper coarse Ricci curvature} $\overline\kappa(x,y)$ are defined respectively by
 \begin{align}  \underline\kappa(x,y):=\liminf_{\lambda\downarrow0}\frac{\kappa_{\lambda}(x,y)}{\lambda}\ \ \ \text{and}\ \ \ \ \overline\kappa(x,y):=\limsup_{\lambda\downarrow0}\frac{\kappa_{\lambda}(x,y)}{\lambda}.\label{ourcurvature}
 \end{align}
If the identity $\ul{\kappa}(x,y) = \ol{\kappa}(x,y)$ holds, then we call this value {\em the coarse Ricci curvature} for $x, y$, denoted by $\kappa(x,y)$. 
 \end{defn}
As shown in Section \ref{sec:general-settings}, we can extend the notion of the upper and lower coarse Ricci curvatures to the setting of submodular transformations \cite[Definition 3.1]{yoshida2019cheeger}.

 \begin{rem}
 \label{lowerboundlowercoarse}
For any $x,y\in V$ with $x\ne y$, the lower coarse Ricci curvature $\ul{\kappa}(x,y)$ exists. More precisely by Lemma \ref{lem:KDfinite} and $d(x,y)\ge 1$, the following estimates hold:
\begin{align*}
\underline\kappa(x,y)&\geq 
-2\diam(H)\vol(V)^{1/2}\max_{z\in V}d_z^{-1/2}>-\infty.
\end{align*}
This implies that $\inf_{x,y}\underline\kappa(x,y)\ge -2\diam(H)\vol(V)^{1/2}\max_{z\in V}d_z^{-1/2}$. 
%
 \end{rem}  
It is not trivial whether the upper coarse Ricci curvature $\ol{\kappa}(x,y)$ is finite or not. However we can prove the following upper estimates.
  \begin{lem}\label{lem:upperfinite}
Let $x$ and $y$ be two distinct vertices. Then for any $f\in \Lip_w^1(V)$ with $f(x)/d_x-f(y)/d_y=d(x,y)$, the following estimate holds:
\begin{align}
     \overline{\kappa}(x,y)&\leq d(x,y)^{-1}\la \calL^0f,\delta_x-\delta_y\ra\label{eq:upper-bound-of-olkappa}.
    \end{align}
where $\calL^0$ is the canonical restriction of $\mathcal{L}$ (see (\ref{eq:L0-and-limit-of-resolvent})). Moreover, the following holds:
\begin{equation}
\label{upperboundofbark}
     \max_{x,y\in V}\overline{\kappa}(x,y)\le 2\diam(H)\vol(V)^{1/2}\max _{z\in V}d_z^{-1/2}<\infty.
\end{equation}
  \end{lem}
   \begin{proof}
Let $\lambda>0$. Then the following estimates hold:
    \begin{align}
     \KD_{\lambda}(x,y)&\geq \la J_{\lambda}f,\delta_x\ra-\la J_{\lambda}f,\delta_y\ra\notag\\
     &=\la J_{\lambda}f-f,\delta_x\ra+d(x,y)-\la J_{\lambda}f-f,\delta_y\ra.\notag
    \end{align}
This implies that the following inequality holds:
    \begin{align}
     \lambda^{-1}\kappa_{\lambda}(x,y)\leq d(x,y)^{-1}\la\lambda^{-1}(J_{\lambda}f-f),\delta_y\ra-\la\lambda^{-1}(J_{\lambda}f-f),\delta_x\ra.\notag
    \end{align}
By taking the superior limit as $\lambda\rightarrow +0$ and using~\eqref{eq:L0-and-limit-of-resolvent}, we have (\ref{eq:upper-bound-of-olkappa}). (\ref{upperboundofbark}) follows from (\ref{eq:upper-bound-of-olkappa}), (\ref{estimateLf}) and $d(x,y)\ge 1$.
   \end{proof}

The following main result means that for any finite connected hypergraphs, the lower and upper coarse Ricci curvatures coincide. 

\begin{thm}[Existence of the coarse Ricci curvature on hypergraphs]
\label{thm:existence-of-limit}
The identity 
 $\ul{\kappa}(x,y)=\ol{\kappa}(x,y)$ holds for any two distinct vertices $x$ and $y$.
\end{thm}
We give a proof of this theorem in Section \ref{sec:existenceRicci} (see Theorem \ref{thm:existenceofRicci}) via linear programming. For convenience of the reader we give a proof in the case of usual graphs as Proposition \ref{prop:campare-LLY} in a more straightforward way than the case for hypergraphs. And we emphasize that only for finite hypergraphs and usual graphs, we can prove the coincidence between the upper and lower Ricci curvature. More general cases, even for infinite hypergraphs, we don't know the coincidence between them.

Next we show a relation between the minimum of the coarse Ricci curvature for any pairs of vertices and that for adjacent vertices. We set $\kappa:=\min_{x,y}\kappa(x,y)=\min_{x\ne y}\kappa(x,y)$.
  \begin{lem}
The identity ${\displaystyle 
\kappa=\min_{x\sim y}\kappa(x,y)}$ holds.
  \end{lem}
  \begin{proof}
It suffices to prove $\kappa \ge \min_{x\sim y}\kappa(x,y)$. Let $x,y\in V$ with $x\ne y$ and set $n:=d(x,y)\ge 1$. Take $\{x_i\}$ be a shortest path connecting $x$ and 
  $y$. Then for any $\lambda>0$, by Proposition~\ref{prop:distance} ~(3) the following inequality holds:  
   \begin{align}
    \kappa_{\lambda}(x,y)&
    \geq 1-\frac{\sum_{i=0}^{n-1}\KD_{\lambda}(x_i,x_{i+1})}{n}
    =\frac{1}{n}\sum_{i=0}^{n-1}\left( 1-\frac{\KD_{\lambda}(x_i,x_{i+1})}{d(x_i,x_{i+1})}\right),\notag
   \end{align}
which implies that the estimates hold:
\[
    \kappa(x,y)\geq \frac{1}{n}\sum_{i=0}^{n-1}\kappa(x_i,x_{i+1})\geq \min_{x\sim y}\kappa(x,y).\qedhere
\]
  \end{proof}

We show another property of the minimum of the coarse Ricci curvatures. We set $\kappa_{\lambda}:=\min_{x,y}\kappa_{\lambda}(x,y)$ for $\lambda>0$.  
  \begin{lem}
  \label{propertyinferior}
The identity ${\displaystyle \liminf_{\lambda\downarrow0 }\kappa_{\lambda}/\lambda=\kappa}$ holds.
  \end{lem}
\begin{proof}
Since $V$ is a finite set, so is $V\times V$. 
We can take $(x_{\lambda},y_{\lambda})\in V\times V$ with $x_{\lambda}\ne y_{\lambda}$ such that $\kappa_{\lambda}(x_{\lambda},y_{\lambda})=\kappa_{\lambda}$. 
We can show that there is a distinct pair $(x_\infty, y_\infty)\in V\times V$ such that 
\[
 \liminf_{\lambda\downarrow 0}\frac{\kappa_{\lambda}(x_\infty,y_\infty)}{\lambda}  
    =\liminf_{\lambda \downarrow 0}\frac{\kappa_{\lambda}}{\lambda}.
\]
Take $(x_0,y_0)\in V\times V$ such that $\kappa= \kappa(x_0, y_0)$. 
Then
by taking the limit inf $\lambda \to 0$, we have 
   \begin{align}
    \kappa&\leq \kappa(x_\infty,y_\infty) = \liminf_{\lambda\downarrow 0}\frac{\kappa_{\lambda}(x_\infty,y_\infty)}{\lambda} 
    =\liminf_{\lambda \downarrow 0}\frac{\kappa_{\lambda}}{\lambda}\notag\\
    &\leq\liminf_{\lambda \downarrow 0}\frac{\kappa_{\lambda}(x_0,y_0)}{\lambda}
    =\kappa.\notag
   \end{align}
   This concludes the proof. \qedhere

  \end{proof}

 %
 %
 \section{Connection of Lin-Lu-Yau's coarse Ricci curvature with ours}\label{sec:comparison-on-graph} 
 
The following proposition says that our coarse Ricci curvature gives a generalization of Lin-Lu-Yau's one \cite{lin2011ricci} on graphs to hypergraphs.
 


 
  \begin{prop}\label{prop:campare-LLY}
Assume that $H=(V,E,w)$ is a weighted undirected graph. Let $x,y\in V$ be two distinct vertices.
Then the identity $\kappa(x,y)=\kappa^{\mathrm{LLY}}(x,y)$ holds, 
  where $\kappa^{\mathrm{LLY}}(x,y)$ is defined by (\ref{eq:LLYdef}).
  \end{prop}

  \begin{proof}
Let $\lambda:= 1 - \alpha\in (0,1)$ and $x,y\in V$ be two distinct vertices. We recall the definitions of the Lin-Lu-Yau's coarse Ricci curvature $\kappa^{\mathrm{LLY}}(x,y)$ (\ref{eq:LLYdef}), the $\alpha$-lazy one $\kappa^{\alpha}(x,y)$ (\ref{alpha-lazy}), our Ricci curvature $\kappa(x,y)$ (\ref{ourcurvature}) and the $\lambda$-coarse one $\kappa_{\lambda}(x,y)$ (\ref{lambda-coarse}). We evaluate the difference of $\kappa^{\alpha}(x,y)$ and $\kappa_{\lambda}(x,y)$.
Since the equation 
   \begin{align}
    \lambda^{-1}\left\vert \kappa^{\alpha}(x,y)-\kappa_{\lambda}(x,y)\right\vert=
\lambda^{-1}d(x,y)^{-1}\vert W_1(m_x^{\alpha},m_y^{\alpha})-\KD_\lambda(x,y) \vert
    \label{eq:difference-of-curvatures}
   \end{align}
   holds, it suffices to evaluate 
$\lambda^{-1}\vert W_1(m_x^{\alpha},m_y^{\alpha})-\KD_\lambda(x,y)\vert$. 
    There exist some potentials to the both $W_1(m_x^{\alpha},m_y^{\alpha})$ (see \cite{Villani:2009oldnew}) and $\KD_{\lambda}(x,y)$ (see Proposition~\ref{prop:nKP}). 
    Let $f^{\alpha}$ be a Kantorovich potential for 
    $(m_x^{\alpha},m_y^{\alpha})$. Noting that $f^{\alpha}$ is $1$-Lipschitz, i.e. $Df\in \Lip_w^1(V)$, by (\ref{KR-duality-graph}), we obtain 
   \begin{align}
    &W_1(m_x^{\alpha},m_y^{\alpha})-\KD_{\lambda}(x,y)\notag\\
    &\leq 
    \left\{\left(I-\lambda \calL\right)f^{\alpha}-J_{\lambda}f^{\alpha}\right\}(x)-\left\{\left(I-\lambda \calL\right)f^{\alpha}-J_{\lambda}f^{\alpha}\right\}(y).\notag
   \end{align} 
Since $J_{\lambda}=(I+\lambda\calL)^{-1}$, if $\lambda$ is sufficiently small, then the Neumann series 
   expansion holds:
   \begin{align}
    J_{\lambda}f=(I-\lambda \mathcal{L})f+\sum_{i=2}^{\infty}(-\lambda\calL)^if.\notag
   \end{align}
   Hence as $\lambda\rightarrow 0$, we have 
   \begin{align}
    &\lambda^{-1}(W_1(m_x^{\alpha},m_y^{\alpha})-\KD_{\lambda}(x,y))\notag\\
    &\leq \sum_{i=2}^{\infty}(-\lambda)^{i-1}\calL^i f^{\alpha}(x)-\sum_{i=2}^{\infty}(-\lambda)^{i-1}\calL^i f^{\alpha}(y)
    \rightarrow 0,\notag
   \end{align}
which implies $\lim_{\lambda\downarrow 0}
	\lambda^{-1}(W_1(m_x^{\alpha},m_y^{\alpha})-\KD_{\lambda}(x,y)) \leq 0$. By exchanging the role of $\KD_{\lambda}$ and $W_1$, we obtain 
   the similar result $\lim_{\lambda\downarrow 0}
	\lambda^{-1}(\KD_{\lambda}(x,y)-W_1(m_x^{\alpha},m_y^{\alpha})) \leq 0$. Consequently, we have 
\begin{align}
    \lim_{\lambda\downarrow 0}\lambda^{-1}\left\vert W_1(m_x^{\alpha},m_y^{\alpha})-\KD_{\lambda}(x,y)\right\vert=0.\label{eq:limit-W1-KD-linear}
    \end{align}
Since the limit $\lim_{\alpha \uparrow 1}\kappa^\alpha(x,y)/(1-\alpha)$
exists by \cite[P.609]{lin2011ricci}, so does the limit $\lim_{\lambda\downarrow 0}\kappa_{\lambda}(x,y)/\lambda$. 
By combining (\ref{eq:difference-of-curvatures}) and (\ref{eq:limit-W1-KD-linear}), we have $\kappa^{\mathrm{LLY}}(x,y)=\kappa(x,y)$.  
  \end{proof}
  \begin{rem}
The argument of the proof of Proposition \ref{prop:campare-LLY} is applicable for other situations. 
Indeed we can show that the Ricci curvature on directed graphs defined by  
  Sakurai et.al \cite[Definition 3.6]{ozawa2020geometric} is same as a 
  modification of our Ricci curvature on directed graphs.
More precisely, since the Laplacian $\Delta$ \cite[Definition 3.6]{ozawa2020geometric} is self-adjoint and non-positive definite operator (\cite[Proposition.~2.4]{ozawa2020geometric}) and the measure appears in their definition 
can be calculated as 
   \begin{align}
    \int_Vf\,d\nu^{\epsilon}_x=(I+\epsilon\Delta)f(x)\notag
   \end{align}
\cite[Lemma~3.1]{ozawa2020geometric}, we can accomplish the 
similar proof as Proposition~\ref{prop:campare-LLY}. 
  \end{rem}

\section{Connections of our Ricci curvature with analytic or geometric properties} 
\label{sec:applications}

\subsection{Eigenvalue of the submodular hypergraph Laplacian} 
\label{subsec:eigenvalues}
We call $\mu \in \R_{>0}$ an eigenvalue of $\calL$ if 
there exists $f \in \R^V$ satisfying $\calL^0f=\mu f$. We can prove that the eigenvalue is bounded by the minimum of the coarse Ricci curvature from below.  

  \begin{thm}\label{thm:eigenvalue}
  Let $\mu$ be an eigenvalue of $\calL$. Then the estimate $\kappa\leq \mu$ holds.  
  \end{thm}
  \begin{proof}
Since $\mu$ is an eigenvalue of $\mathcal{L}$, there exists $f\in \R^V$ such that $\calL^0f=\mu f$. By multiplying some constant if necessary and Lemma \ref{lem:homogenity-of-Laplacian}, we may assume $f\in \Lip_w^1(V)$. Moreover, without loss of generality, we may assume that 
$f(x)/d_x-f(y)/d_y=d(x,y)$ holds for some $x,y \in V$. By Lemma \ref{lem:upperfinite} and Theorem \ref{thm:existence-of-limit}, the estimates hold:
   \begin{align*}
    \kappa&\le \kappa(x,y)
    \le d(x,y)^{-1}\left(\langle \calL^0f,\delta_x\rangle-\langle \calL^0f,\delta_y\rangle\right)\\
    &=\frac{\mu}{d(x,y)}\left(\frac{f(x)}{d_x}-\frac{f(y)}{d_y}\right)=\mu.\notag
   \end{align*}
  \end{proof}
 \begin{rem}
  The same conclusion of Theorem \ref{thm:eigenvalue} is proven if $\inf_{x\sim y}\ul{\kappa}(x,y)\leq \kappa$ for suitable settings(infinite hypergraphs, submodular transformation etc.). 
 \end{rem}

\subsection{Gradient estimate of the heat flow}
\label{subsec:gradient-estimate}
Next we prove a relation between a lower bound of our Ricci curvature and a gradient estimate of the heat flow. 
  \begin{thm}
\label{gradient estimate}
Let $\kappa_0\in \R$. Assume that the inequality $\kappa\ge \kappa_0$ holds. Then any $x,y\in V$, $f\in \Lip_w^1(V)$ and $t>0$, the following inequality holds:  
   \begin{align}
    \frac{e^{-t\mathcal{L}} f(x)}{d_x}-\frac{e^{-t\mathcal{L}}f(y)}{d_y}\leq e^{-\kappa_0 t}d(x,y).\notag
   \end{align}
  \end{thm}
 \begin{proof}
Let $\lambda>0$. The definitions of $\KD_\lambda$ and $\kappa_\lambda$ give 
  \begin{align}
   \frac{J_{\lambda}f(x)}{d_x}-\frac{J_{\lambda}f(y)}{d_y}
   \leq \KD_{\lambda}(x,y)\leq (1-\kappa_{\lambda})d(x,y),\notag
  \end{align}
which implies that $(1-\kappa_{\lambda})^{-1}J_{\lambda}f\in \Lip_w^1(V)$. In the similar calculation with (1) of Lemma \ref{lem:homogenity-of-resolvent}, we have   
  \begin{align}
   &\frac{J_{\lambda}^2f(x)}{d_x}-\frac{J_{\lambda}^2f(y)}{d_y}
   =(1-\kappa_{\lambda})\left(\la J_{\lambda}\left(\frac{J_{\lambda}f}{1-\kappa_{\lambda}}\right),\delta_x\ra-\la J_{\lambda}\left(\frac{J_{\lambda}f}{1-\kappa_{\lambda}}\right),\delta_y\ra\right)\notag\\
   &\leq (1-\kappa_{\lambda})^2d(x,y).\notag
  \end{align}
Repeating the similar calculation implies that 
$\la J_{\lambda}^nf,\delta_x\ra-\la J_{\lambda}^nf,\delta_y\ra\leq (1-\kappa_{\lambda})^nd(x,y)$ holds for any $n\in \mathbb{N}\cup \{0\}$. 
Therefore, by \eqref{eq:heat-and-resolvent} and Lemma \ref{propertyinferior}, we have 
  \begin{align}
   \frac{e^{-t\mathcal{L}}f(x)}{d_x}-\frac{e^{-t\mathcal{L}}f(y)}{d_y}&=\lim_{\lambda\downarrow0}\la J_{\lambda}^{[t/\lambda]}f,\delta_x\ra-\la J_{\lambda}^{[t/\lambda]}f,\delta_y\ra\notag\\
   &\leq \liminf_{\lambda\downarrow0}(1-\kappa_{\lambda})^{[t/\lambda]}d(x,y)\notag\\
   &\leq\liminf_{\lambda\downarrow0}e^{-\frac{\kappa_{\lambda}}{\lambda}[t/\lambda]\lambda}d(x,y)\notag\\
   &=e^{-\kappa t}d(x,y)\le e^{-\kappa_0 t}d(x,y).\notag 
  \end{align}
Here the second inequality follows from the inequality 
 $(1+x)^t\le e^{xt}$ for any $\vert x\vert<1$ and $t>0$. 
%
 \end{proof}
  \begin{rem}
   The same conclusion of Theorem \ref{gradient estimate} is proven if $\inf_{x\sim y}\ol{\kappa}(x,y)\leq \kappa$ for suitable settings(infinite hypergraphs, submodular transformation etc.).
  \end{rem}
 
\subsection{Bonnet-Myers diameter bound under positive Ricci curvature} 
\label{subsec:diameter-bound}
We prove a geometric consequence (Bonnet-Myers diameter bound) under the Ricci curvature being positive. The following gives a generalization of \cite[Theorem 4.1]{lin2011ricci} and \cite[Proposition 23]{ollivier2009ricci} to the case of hypergraphs.
  \begin{thm}[Bonnet-Myers diameter bound]
  \label{Bonnet-Myers}
Assume that $\kappa\ge \kappa_0>0$ holds. Then the following holds:
   \begin{align}
   \label{Bonnet-Myersineq}
    \diam(H) \leq 2\kappa_0^{-1}.
   \end{align}
  \end{thm}
  \begin{proof}
Let $x,y\in V$ be two distinct vertices which satisfy $d(x,y)=\di(H)$. By Lemma \ref{lem:upperfinite} and Theorem \ref{thm:existence-of-limit}, there exists $f\in\Lip_w^1(V)$ with $\la f,\delta_x-\delta_y\ra=d(x,y)$ such that $\kappa(x,y)\le d(x,y)^{-1}\la\calL^0f,\delta_x-\delta_y\ra$. By (\ref{eq:upperbound-of-L(f)}), the estimate $\vert \calL^0f(z)\vert\leq d_z$ holds for any $z\in V$. By the assumption, the following inequalities hold:
   \begin{align}
    0<\kappa_0&\le \kappa\leq \kappa(x,y)\le d(x,y)^{-1}\la\calL^0f,\delta_x-\delta_y\ra\notag\\
    &\le 2d(x,y)^{-1}=2\di(H)^{-1}.\notag
   \end{align}
  \end{proof}
  \begin{rem}
   The same conclusion of Theorem \ref{Bonnet-Myers} is proven if $\inf_{x\sim y}\ol{\kappa}(x,y)\leq \kappa$ for suitable settings(submodular transformation etc.).
  \end{rem}
The second author et.al \cite[Theorem 1.1]{Kitabeppu_Matsumoto} proved Cheng's maximal diameter theorem, which means that if the equality of (\ref{Bonnet-Myersineq}) holds, then there exists a pair of vertices $x,y$ with $d(x,y)=\diam(H)$ such that all points on a geodesic from $x$ to $y$. 

 %
 %
 \section{Examples}
 \label{sec:examples} 
In this section, we calculate the values of our curvature for several hypergraphs. Let $H=(V,E,w)$ be a weighted undirected hypergraph. The key formula for calculations is Moreau's theorem (\ref{Moreau's theorem}).
 
\begin{example}\label{ex:three-vertices} 
We consider the case where  
 $V:=\{x,y,z\}$, $E:=\{xy,yz,zx,xyz\}$, and $w(e):=1$ for 
 any $e \in E$. 
 We calculate the coarse Ricci curvature $\kappa(x,y)$

First we calculate the $\lambda$-nonlinear Kantrovich difference $\KD_{\lambda}(x,y)$ for a sufficiently small $\lambda > 0$. 
Let $f\in \Lip_w^1(V)$. 
We set the values $f(x) =: 3\alpha, f(y) =: 3\beta$, and $f(z) =: 3\gamma$. In the similar argument as the proof of Proposition \ref{prop:finitenorm}, we may assume that $\beta=0$. We divide our argument into the four cases:\ (1)\ $\alpha > \gamma > 0$,\ (2)\ $\gamma > \alpha > 0$,\ (3)\ $\gamma = 0$,\ (4)\ $\gamma = \alpha$,\ (5)\ $\alpha>0>\gamma$. 

We set $g:=J_{\lambda}f$. Moreover we divide the cases for the values of $g$. We remark that $\alpha,\gamma\leq 1$ holds since $f\in \Lip_w^1(V)$. Since $g=J_{\lambda}f\rightarrow f$ as $\lambda\rightarrow 0$ due to (\ref{eq:bound-by-operator-norm}), we write $g(x)=3\alpha+3a$, $g(y)=3b$, $g(z)=3\gamma+3c$, where $|a|,|b|$ and $|c|$ are sufficiently small. We define $F:\R^V\rightarrow \R$ as
\begin{align}
 F(g):=\frac{1}{2\lambda}\Vert f-g\Vert^2+Q(D^{-1}g).\notag
\end{align} 
\noindent
\underline{(1) $\alpha>\gamma>0$.} 
Since $g=J_{\lambda}f$ is closed to $f$, we may assume $\alpha+a>\gamma+c>b$. 
Then the normalized Laplacian $\ca{L}$ of $g$ is uniquely determined and the following hold: 
 \begin{align}
  &\sfb^{\top}_{xy}(D^{-1}g)=\alpha+a-b,\;\sfb^{\top}_{xz}(D^{-1}g)=\alpha+a-\gamma-c,\;\sfb^{\top}_{yz}(D^{-1}g)=\gamma+c-b,\notag\\
  &\sfb_{xyz}^{\top}(D^{-1}g)=\alpha+a-b.\notag
 \end{align}
Hence, we have 
 \begin{align}
  F(g)=\frac{1}{2\lambda}\left(3a^2+3b^2+3c^2\right)+\frac{1}{2}\left(2(\alpha+a-b)^2+(\gamma+c-b)^2+(\alpha+a-\gamma-c)^2\right).\notag
 \end{align}
Let $r:=\lambda^{-1}>0$ be sufficiently large. Since $J_{\lambda}f$ is a critical point for $F$, $\partial_aF=\partial_bF=\partial_cF=0$, which is equivalent to 
 \begin{align}
  \begin{pmatrix}
   3(1+r)&-2&-1\\
   -2&3(1+r)&-1\\
   -1&-1&2+3r
  \end{pmatrix}
  \begin{pmatrix}
  a\\ b\\ c
  \end{pmatrix}=
  \begin{pmatrix}
   -3\alpha+\gamma\\ 2\alpha+\gamma\\ \alpha-2\gamma r
  \end{pmatrix}.\notag
 \end{align} 
 This can be solved and we see that 
 $(a,b,c )^\top$ is equal to 
 \begin{align}
  &\frac{1}{9r(1+r)(3r+5)}\cdot \notag\\ 
 & \begin{pmatrix}
   3(1+r)(2+3r)-1&2(2+3r)+1&2+3(1+r)\\
   2(2+3r)+1&3(1+r)(2+3r)-1&3(1+r)+2\\
   2+3(1+r)&3(1+r)+2&9(1+r)^2-4
  \end{pmatrix}
  \begin{pmatrix}
   -3\alpha+\gamma\\ 2\alpha+\gamma\\ \alpha-2\gamma r
  \end{pmatrix}.\notag
 \end{align}
Since the inner product $\la J_{\lambda}f,\delta_x-\delta_y\ra$ is represented as $\alpha+a-b$, we have  
 \begin{align}
  &\la J_{\lambda}f,\delta_x-\delta_y\ra
  =\alpha+a-b
  =\frac{3\alpha r}{3r+5}\leq \frac{3r}{3r+5}.\notag
 \end{align}
Here the last inequality follows from $\alpha\le 1$ and the equality is attained when $\alpha=1$.\\
 
\noindent \underline{(2) $\gamma>\alpha>0$.}
By a similar argument as above, we may assume $\gamma+c>\alpha+a>b$. 
Then the following holds: 
 \begin{align}
  &\sfb^{\top}_{xy}(D^{-1}g)=\alpha+a-b,\;\sfb^{\top}_{xz}(D^{-1}g)=\gamma+c-\alpha-a,\;\sfb^{\top}_{yz}(D^{-1}g)=\gamma+c-b,\notag\\
  &\sfb^{\top}_{xyz}(D^{-1}g)=\gamma+c-b,\notag
 \end{align}
this implies 
 \begin{align}
  F(g)=\frac{1}{2\lambda}(3a^2+3b^2+3c^2)+\frac{1}{2}\left((\alpha+a-b)^2+2(\gamma+c-b)^2+(\gamma+c-\alpha-a)^2\right).\notag
 \end{align}
 From $\partial_aF=\partial_bF=\partial_cF=0$, we obtain 
 \begin{align}
  \begin{pmatrix}
   3r+2&-1&-1\\
   -1&3(1+r)&-2\\
   -1&-2&3(1+r)
  \end{pmatrix}
  \begin{pmatrix}
   a\\ b\\ c
  \end{pmatrix}=
  \begin{pmatrix}
   -2\alpha +\gamma\\\alpha+2\gamma\\ \alpha-3\gamma 
  \end{pmatrix}.\notag
 \end{align}
This  equation can also be solved and we see that $(a,b,c)^\top$ is equal to 
 \begin{align}
 & \frac{1}{9r(1+r)(3r+5)}\cdot \notag\\
 & \begin{pmatrix}
   9(1+r)^2-4&3(1+r)+2&3(1+r)+2\\
   3(1+r)+2&3(1+r)(2+3r)-1&2(2+3r)+1\\
   3(1+r)+2&2(2+3r)+1&3(1+r)(2+3r)-1
  \end{pmatrix} \begin{pmatrix}
   -2\alpha+\gamma\\ \alpha+2\gamma\\ \alpha-3\gamma 
  \end{pmatrix}.\notag
 \end{align}
Then, we have 
 \begin{align*}
  \la J_{\lambda}f,\delta_x-\delta_y\ra &=\alpha+a-b
  =\frac{r\left(\alpha(3r+5)-\gamma\right)}{(1+r)(3r+5)}.\notag
 \end{align*}
 
\noindent 
\underline{(3) $\gamma=0$.} By the symmetry of $H$ and $f$, we have $b=c$, which implies 
 \begin{align}
  &\sfb^{\top}_{xy}(D^{-1}g)=\alpha+a-b,\;\sfb^{\top}_{xz}(D^{-1}g)=\alpha+a-b,\;\sfb^{\top}_{yz}(D^{-1}g)=0,\notag\\
  &\sfb^{\top}_{xyz}(D^{-1}g)=\alpha+a-b.\notag
 \end{align}
 Thus we have 
 \begin{align}
  F(g)=\frac{1}{2\lambda}\left(3a^2+6b^2\right)+\frac{3}{2}(\alpha+a-b)^2.\notag
 \end{align}
 The identities $\partial_aF=\partial_bF=0$ give
 \begin{align}
  \begin{pmatrix}
   1+r&-1\\
   -1&1+2r
  \end{pmatrix}
  \begin{pmatrix}
  a\\b
  \end{pmatrix}=
  \begin{pmatrix}
   -\alpha \\
   \alpha
  \end{pmatrix}
  \Longleftrightarrow 
  \begin{pmatrix}
   a\\b
  \end{pmatrix}
  =\frac{\alpha}{2r+3}
  \begin{pmatrix}
   -2\\
   1
  \end{pmatrix}.
  \notag
 \end{align}
Hence, we have 
 \begin{align}
  \la J_{\lambda}f,\delta_x-\delta_y\ra=\alpha+a-b=\frac{2\alpha r}{2r+3}\leq \frac{2r}{2r+3}.\notag
 \end{align}
 The last inequality follows from $\alpha\le 1$ and the identity is attained when $\alpha=1$. 
 
\noindent
\underline{(4) $\gamma=\alpha$. }
We have $a=c$ similarly as the above, which implies
 \begin{align}
  &\sfb^{\top}_{xy}(D^{-1}g)=\alpha+a-b,\;\sfb^{\top}_{xz}(D^{-1}g)=0,\;\sfb^{\top}_{yz}(D^{-1}g)=\alpha+a-b,\notag\\
  &\sfb^{\top}_{xyz}(D^{-1}g)=\alpha+a-b.\notag
 \end{align}
Thus we have 
 \begin{align}
  F(g)=\frac{1}{2\lambda}(6a^2+3b^2)+\frac{3}{2}(\alpha+a-b)^2.\notag
 \end{align} 
The 
equations $\partial_aF=\partial_bF=0$ can be written as  
 \begin{align}
  \begin{pmatrix}
   1+2r&-1\\
   -1&1+r
  \end{pmatrix}
  \begin{pmatrix}
   a\\b
  \end{pmatrix}=
  \begin{pmatrix}
   -\alpha \\ \alpha
  \end{pmatrix}
  \Longleftrightarrow \begin{pmatrix}
   a\\b
  \end{pmatrix}=\frac{\alpha}{2r+3}
  \begin{pmatrix}
   -1\\2
  \end{pmatrix}.
 \notag
 \end{align}
%
Consequently, we have 
 \begin{align}
  \la J_{\lambda}f,\delta_x-\delta_y\ra=\alpha+a-b=\frac{2\alpha r}{2r+3}\leq \frac{2r}{2r+3}.\notag
 \end{align}
 The last inequality follows from $\alpha\leq 1$ and the identity is attained when $\alpha=1$.
 
\noindent 
\underline{(5) $\alpha>0>\gamma$.} Since $f\in \Lip_w^1(V)$, $\alpha-\gamma\leq 1$ and $\alpha<1$. Then the following holds:
\begin{align}
  &\sfb^{\top}_{xy}(D^{-1}g)=\alpha+a-b,\;\sfb^{\top}_{xz}(D^{-1}g)=\alpha+a-\gamma-c,\;\sfb^{\top}_{yz}(D^{-1}g)=b-\gamma-c,\notag\\
  &\sfb^{\top}_{xyz}(D^{-1}g)=\alpha+a-\gamma-c,\notag
 \end{align}
which implies 
 \begin{align}
  F(g)=\frac{1}{2\lambda}\left(3a^2+3b^2+3c^2\right)+\frac{1}{2}\left((\alpha+a-b)^2+2(\alpha+a-\gamma-c)^2+(b-\gamma-c)^2\right).\notag
 \end{align}
 In the same manner as before, we obtain 
 \begin{align}
  \begin{pmatrix}
   3(1+r)&-1&-2\\
   -1&3r+2&-1\\
   -2&-1&3(1+r)
  \end{pmatrix}
  \begin{pmatrix}
   a\\b\\c
  \end{pmatrix}=
  \begin{pmatrix}
   -3\alpha+2\gamma\\
   \alpha+\gamma\\
   2\alpha-3\gamma
  \end{pmatrix},\notag
 \end{align}
whici implies that $(a \ b \ c)^\top$ is equal to 
 \begin{align}
  \frac{1}{9r(1+r)(3r+5)}
  \begin{pmatrix}
   3(1+r)(3r+2)-1&3(1+r)+2&2(2+3r)+1\\
   3(1+r)+2&9(1+r)^2-4&3(1+r)+2\\
   2(2+3r)+1&3(1+r)+2&3(1+r)(2+3r)-1
  \end{pmatrix}.\notag
 \end{align}
 Finally we have 
 \begin{align}
  \la J_{\lambda}f,\delta_x-\delta_y\ra=\alpha+a-b 
  =\frac{r}{(1+r)(3r+5)}\left\{\alpha(3r+4)+\gamma \right\}.\notag
 \end{align}

By comparing the values of $\la J_{\lambda}f,\delta_x-\delta_y\ra$ for the above all cases, we can show that 
the values for the case (1) are less than or equal to 
$2r/(2r+3)$, which is attained for the cases (3) and (4) with $\alpha=1$. Thus, it suffices to compare the cases (2), (3), and (5). 
We can calculate the differences as 
 \begin{align*}
  (3)-(2)
  &=\frac{2r}{2r+3}-\frac{r\left(\alpha(3r+5)-\gamma\right)}{(1+r)(3r+5)}
  \geq \frac{r\left(2(1-\alpha)r+2-3\alpha\right)}{(2r+3)(1+r)}\ge 0, 
  \text{ and } \notag \\
 (3)-(5)
  &=\frac{2r}{2r+3}-\frac{r}{(1+r)(3r+5)}\left\{\alpha(3r+4)+\gamma \right\}\notag\\
  &\geq \frac{2r}{2r+3}-\frac{\alpha r}{1+r}=\frac{r\{2(1-\alpha)r+2-3\alpha\}}{(2r+3)(1+r)}\geq 0.\notag
 \end{align*}
Here the most right hand sides are non-negative, since
$r=\lambda^{-1}$ is sufficiently large and $1\geq \gamma>\alpha$ in the case (2) and $\alpha<1$, $\gamma<0$ in the case (5). 
Thus, we have 
 \begin{align}
  \KD_{\lambda}(x,y)=\frac{2\lambda^{-1}}{2\lambda^{-1}+3}.\notag
 \end{align}
Consequently, the coarse Ricci curvature $\kappa(x,y)$ exists and becomes
 \begin{align}
  \kappa(x,y)=\lim_{\lambda\rightarrow +0}\frac{1}{\lambda}\left(1-\frac{\KD_{\lambda}(x,y)}{d(x,y)}\right)=\frac{3}{2}.\notag
 \end{align}
 \end{example}

 \begin{rem}
 We conjecture that if one consider the hypergraph $H =(V, E, w)$ such that  
 $|V| = n$, $E = 2^V\setminus \{\emptyset, V\} $ and 
 $w(e) = 1$ for any $e \in E$, then the $\lambda$-nonlinear Kantorovich potential $f$ satisfies that for $x\in V$, $f(x)=d_x$ and  $f(z)=0$ ($z\neq x$).  
 \end{rem}


We conjecture that the following formula holds, which enables us to easily calculate our curvatures. The similar formula was proved in the case of usual graphs \cite[Theorem 2.1]{MW:2019Oll}.
\begin{conj}
\label{characterization}
For any two distinct vertecices $x$ and $y$, the following holds:
  \[
\kappa(x,y)=\frac{\inf\left\{\la\calL^0f,\delta_x-\delta_y\ra\;;\;f\in\Lip_w^1(V),\;\la f,\delta_x-\delta_y\ra=d(x,y)\right\}}{d(x,y)}.
  \]
\end{conj}

\begin{example}
We consider the case where $V=\{x,y,z\}$, $E=\{ e = \{ x,y,z \} \}$, and $w_e=1$.
We consider $f:V\rightarrow\R$ such that $f(x)=1$, $f(y)=0$, $f(z)=0$. Then, we have $\calL^0f(x)=1$, $\calL^0f(y)=-1/2$, $\calL^0f(z)=-1/2$. Lemma \ref{lem:upperfinite} gives
   \begin{align}
    \kappa(x,y)\le \la \calL^0f,\delta_x-\delta_y\ra=1-(-1/2)=3/2. \notag
   \end{align}
Actually we can prove that $\kappa(x,y)=3/2$.  
\end{example}

\begin{example}[complete hypergraph]
We consider the case where $V=\{v_1,v_2,\ldots,v_n\}$, $E=2^V\setminus\{\{v_1\},\ldots,\{v_n\},\emptyset\}$, and $w_e=1$. 
Then we have $|E|=2^n-n-1$ and $d_x=2^{n-1}-1=:d$ for any $x \in V$. 
We count the number of hyperedges $e \in E$ including $v_1$ and $v_2$. 
The number of such $e$ satisfying $|e| = k$ is ${n-2 \choose k-2}$. 

Let $f:V\rightarrow \R$ be the function satisfying $f(v_1)=d$ and $f(v_i)=0$ $(i=2,\ldots,n)$. 
Then, we have $\calL^0f(v_1)=d$. Moreover, for $e$ including $v_1$ and $v_2$ such that $\#e = k$, 
we may choose $\delta_{v_1}-(k-1)^{-1}\sum_{i\geq 2,v_i\in e}\delta_{v_i}$, $v_1,v_2\in e$ as $\sfb_e$. Thus, we have 
   \begin{align}
    \calL^0f(v_2)&=-\sum_{k=2}^n\frac{1}{k-1}{n-2 \choose k-2}
    =-\frac{2^{n-1}-1}{n-1}.\notag
   \end{align}
Thus Lemma \ref{lem:upperfinite} gives 
   \begin{align}
    \kappa(v_1,v_2)&\leq \frac{1}{d}\left(\calL^0f(v_1)-\calL^0f(v_2)\right)
    =\frac{n}{n-1}.\notag
   \end{align}
From this observation, we conjecture that the curvature of the complete 
hypergraph $H$ with $|V|=n$ is $n/(n-1)$. 
This prediction agrees with calculation in Example~\ref{ex:three-vertices}. 
\end{example}

\begin{example}[Negatively curved hypergraph]
 We consider 9 points hypergraph, $V=\{w_1,w_2,x,y,z,u_1,u_2,u_3,u_4\}$, $E=\{e_0=\{x,y,z\},e_1=\{w_1,w_2,x\},e_2=\{y,u_1,u_2\},e_3=\{z,u_3,u_4\}\}$, and $w_{e_i}=1$ for $i=0,1,2,3$. Let $f:V\rightarrow \R$ be a function defined by $f(w_1)=f(w_2)=f(x)=2$, $f(y)=f(z)=0$, and $f(u_j)=-1$ for $j=1,2,3,4$. Then $f$ is a weighted 1-Lipschitz function. By using a calculation like above, we have $\calL^0f(x)-\calL^0f(y)=-1/2$. Thus Lemma \ref{lem:upperfinite} implies $\kappa(x,y)\leq -1/4$. 
\end{example}


\section{Existence of the coarse Ricci curvature on hypergraphs}\label{sec:existenceRicci}

The purpose of this section is to prove Theorem \ref{thm:existence-of-limit}, i.e. the following theorem.
\begin{thm} \label{thm:existenceofRicci}
The coarse Ricci curvature 
\begin{align}
  \kappa(x_0,y_0)=\lim_{\lambda\downarrow0}\frac{1}{\lambda}\left( 1-\frac{\KD_{\lambda}(x_0,y_0)}{d(x_0,y_0)}\right) \label{eqn:existenceofRicci}
\end{align}
along with $x_0, y_0 \in V$ on a hypergraph $H$ exists, where $\KD_{\lambda}(x_0,y_0)$ is the $\lambda$-nonlinear Kantorovich difference, given in Definition \ref{def:NKD}. 
\end{thm}

Let us consider a {\it generalized} hypergraph $H=(V,E,w,d)$ consisting of a finite set $V$, a set $E$ of nonempty subsets of $V$, a function $w : E \to \mathbb{R}_{>0}$, and a function $d : V \to \mathbb{R}_{>0}$. 
For a while, the condition $d_x = \sum_{e \ni x} w_e$ is {\it not} imposed.  
For simplicity, we set $|V|:=n$ and $V:=\{1,2,\dots,n\}$.
The vector space $\mathbb{R}^n$ of real valued functions on $V$ can be expressed as the disjoint union
\[
\mathbb{R}^n = \bigsqcup_{\rho \in R_n} U_{\rho}, 
\]
where $R_n$, corresponding to the set of orderings, is defined by
\[
R_n=\{ \rho=(\sigma,\tau) \in \mathfrak{S}_n \times \{ 0, 1 \}^{n-1} \mid \sigma(i)<\sigma(i+1)\text{ if } \tau(i)=0 \}, 
\]
and for $\rho=(\sigma,\tau) \in R_n$, the component $U_{\rho}$ is defined by
\[
U_\rho=\left\{(f_1,\dots,f_n)^\top \in \mathbb{R}^n \middle| \begin{array}{l} f_{\sigma(i)} = f_{\sigma(i+1)} \text{ if } \tau(i)=0 \\ f_{\sigma(i)} < f_{\sigma(i+1)} \text{ if } \tau(i)=1 \end{array}
\right\}.
\]
Hence two vectors $f=(f_1,\dots,f_n)^\top, g=(g_1,\dots,g_n)^\top \in \mathbb{R}^n$ belong to the same component $U_{\rho}$ for some $\rho \in R_n$ 
if and only if the elements of $f$ and $g$ are in the same order, that is, 
$\mathrm{sgn}(f_x-f_y)=\mathrm{sgn}(g_x-g_y)$ for any $x,y \in V$ with $\mathrm{sgn} : \mathbb{R} \to \{-1,0,1\}$ given by
\[
\mathrm{sgn}(r)=
\left\{
\begin{array}{cc}
1 & (r>0), \\
0 & (r=0), \\
-1 & (r<0). \\
\end{array}
\right. 
\]
We notice that the dimension of $U_\rho$ is $1+\sum_{i=1}^{n-1} \tau(i)$ for $\rho=(\sigma,\tau) \in R_n$.
Let $K:=\mathbb{Q}(\{w_e\}_{e \in E},\{d_x\}_{x \in V}) \subset \mathbb{R}$ be the subfield of $\mathbb{R}$ generated by 
$\{w_e\}_{e \in E}$ and $\{d_x\}_{x \in V}$, and consider the field $K(z)$ of rational functions in $z$ with coefficients in $K$. 
Moreover, let $G=G_z=G_{H,z} : \mathbb{R}^n \to 2^{\mathbb{R}(z)^n}$ be the (multi-valued) function defined by
\begin{align} \label{eqn:multivaluedGfunction}
Gf=G_zf:=(D+z L)(f) \qquad (f \in \mathbb{R}^n), 
\end{align}
which defines $G_\lambda : \mathbb{R}^n \to 2^{\mathbb{R}^n}$ for any $\lambda>0$, 
where $D=\mathrm{diag}(d_1,\dots,d_n)$ and $L\colon \RR^n \to 2^{\RR^n}$ is the hypergraph Laplacian given in (\ref{eq:def-of-Laplacian}). 
\subsection{Piecewise linear inverse} \label{subsection:pli}
Let $H$ be a generalized hypergraph. 
We will show the following proposition. 
\begin{prop} \label{prop:piecewiselinear}
For any $\rho \in R_n$, there exists a symmetric matrix $N_{\rho,z} \in M_n(K(z))$ such that $N_{\rho,\lambda}$ has non-negative entries for any $\lambda>0$ 
and $(N_{\rho,\lambda} \circ G_{\lambda})(f)=f$ holds for any $f \in U_{\rho}$. 
\end{prop}
In order to prove the proposition, let us prepare the following notations. 
\begin{itemize}
\item For each $e \in E$, put $f_e^+:=\max_{x \in e} f_x$ and $f_e^-:=\min_{x \in e} f_x$ with $f=(f_1,\dots,f_n)^\top \in \mathbb{R}^n$. 
\item For each $\rho \in R_n$, let $V=\sqcup_{I \in \mathcal{I}_{\rho}} I$ be a unique decomposition  of $V=\{1,\dots,n\}$ such that 
each $f=(f_1,\dots,f_n)^\top \in U_{\rho}$ satisfies $f_x=f_y$ if and only if $x,y \in I$ for some $I \in \mathcal{I}_{\rho}$. 
The decomposition is independent of the choice of $f \in U_{\rho}$. 
\item For each $e \in E$ and $\rho \in R_n$, let $I_{e,\rho}^{\pm} \subset V$ be the set of points $x \in V$ with $f_x=f_e^{\pm}$ for $f\in U_{\rho}$, which is also independent of the choice of $f \in U_{\rho}$. 
In particular, we have $I_{e,\rho}^{\pm} \in \mathcal{I}_{\rho}$. Moreover, either $I_{e,\rho}^+ \cap I_{e,\rho}^-=\emptyset$ or $e \subset I_{e,\rho}^+ = I_{e,\rho}^-$ holds. 
\item For each $e \in E$ and $\rho \in R_n$, set 
\[
\mathcal{M}_{e,\rho}:= \text{Conv}(\{S_{xy} \;;\; x \in e \cap I_{e,\rho}^+, y \in e \cap I_{e,\rho}^-\}), 
\]
where $S_{xy} \in M_n(\mathbb{Q})$ is a symmetric matrix, given by
\[
S_{xy}= I_{xx}-I_{xy}-I_{yx}+I_{yy}=
\begin{blockarray}{cccccc}
& & x & & y & & \\[-8pt]
& & \smallsmile & & \smallsmile & & \\[-5pt]
\begin{block}{r(ccccc)}
&  &  &  &  &  &\\
x\ ) &   & 1 &  & -1 &  & \\
&  &  &  &  &  &  \\
y\ ) &  & -1 &   & 1 &  & \\
&  &  &   &  & & \\
\end{block}
\end{blockarray}
\]
with $I_{xy}=\left( \delta_{x}(i) \cdot \delta_y(j) \right)_{1 \le i,j \le n} \in M_n(\mathbb{Q})$
(here $S_{xy}=0$ if $x=y$). 
In particular, each element of $\mathcal{M}_{e,\rho}$ is symmetric. 
\end{itemize}
Now we notice that the hypergraph Laplacian $L\colon \RR^n \to 2^{\RR^n}$ defined in (\ref{eq:def-of-Laplacian}) is expressed as 
\begin{align} \label{eqn:Ldecomposition}
L(f)=\sum_{e\in E}\omega_e L_e(f) \quad \text{with} \quad L_e(f):=\left\{\ttb_e(\ttb_e^{\top}f)\;;\;\ttb_e\in\text{argmax}_{\ttb\in B_e}\ttb^{\top}f\right\}, 
\end{align}
where $B_e$ is given in (\ref{eq:base-polytope-hyper}). 
Here the sum of subsets $A, B \subset \mathbb{R}^n$ stands for the Minkowski sum: $A+B:=\{a+b \in \mathbb{R}^n \;;\; a \in A, b \in B\}$, 
and the multiplication of $A \subset \mathbb{R}^n$ by a scalar $c \in \mathbb{R}$ means $cA:=\{c a \in \mathbb{R}^n \;;\; a \in A \}$. 
Hence the restriction of $L_e$ on each component $U_{\rho}$ is calculated as 
\begin{align} \label{eqn:Leexpression}
L_e|_{U_{\rho}}(f)=(f_e^+-f_e^-) \text{Conv}(\{ \delta_x - \delta_y \;;\; x \in e \cap I_{e,\rho}^+, y \in e \cap I_{e,\rho}^- \})=\mathcal{M}_{e,\rho}f. 
\end{align}
If $I_{e,\rho}^+ = I_{e,\rho}^-$, then one has $e \subset I_{e,\rho}^+ = I_{e,\rho}^-$ and $L_e(f)=\mathcal{M}_{e,\rho}f=0$ for any $f \in U_{\rho}$. 

\begin{proof}[Proof of Proposition \ref{prop:piecewiselinear}] We divide the proof into three steps. 
\\
\underline{Step 1} Assume that $\rho \in R_n$ satisfies $\mathcal{I}_\rho=\{\{1,2,\dots,n\}\}$, which means that $f=(f_1,\dots,f_n)^{\top} \in U_\rho$ satisfies $f_x = f_y$ for any $x, y \in V$. 
In this case, one has $I_{e,\rho}^+ = I_{e,\rho}^-$ and thus $L_e|_{U_\rho}=\mathcal{M}_{e,\rho}=0$ for any $e \in E$. Therefore, $G_z|_{U_\rho}=D+z \sum_{e \in E} w_e \mathcal{M}_{e,\rho}=D$ 
is a single matrix and $N_{\rho,z}=D^{-1}$ is a symmetric matrix with $N_{\rho,\lambda}=D^{-1}$ having non-negative entries for any $\lambda>0$. 
\\
\underline{Step 2} Assume that $\rho \in R_n$ satisfies $\mathcal{I}_\rho=\{\{1\},\{2\},\dots,\{n\} \}$, which means that $f=(f_1,\dots,f_n)^{\top} \in U_\rho$ satisfies $f_x \neq f_y$ for any $x \neq y \in V$. 
In this case, one has $\# I_{e,\rho}^+ = \# I_{e,\rho}^-=1$ and thus $L_e|_{U_{\rho}}=\mathcal{M}_{e,\rho}$ is a single-valued function for any $e \in E$.  
Moreover, $G_z|_{U_\rho}=D+z\sum_{e \in E} w_e \mathcal{M}_{e,\rho}$ is a symmetric matrix such that $G_\lambda|_{U_\rho}$ has positive diagonal entries and non-negative off-diagonal entries for any $\lambda>0$, 
and it satisfies 
\[
\left(D+z \sum_{e \in E} w_e \mathcal{M}_{e,\rho}\right) \left( \begin{array}{c} 1 \\ \vdots \\1 \end{array} \right) = \left( \begin{array}{c} d_1 \\ \vdots \\ d_n \end{array} \right). 
\] 
Hence it is known (see e.g. \cite[Theorem 6.34]{zhan2013matrix}) that there exists an inverse matrix $N_{\rho,z}=(D+z\sum_{e \in E} w_e \mathcal{M}_{e,\rho})^{-1} \in M_n(K(z))$, which is symmetric, and 
$N_{\rho,\lambda}$ has non-negative entries for any $\lambda>0$. 
\\
\underline{Step 3} We prove the proposition by induction on $n$. The proposition for the case $n=1$ can be proved from Step 1 (or Step 2). 
Assume that the proposition holds for $n<m$, and consider the case $n=m$. If $\rho \in R_m$ satisfies $\mathcal{I}_\rho=\{\{1,2,\dots,m\}\}$ or $\mathcal{I}_\rho=\{\{1\},\{2\},\dots,\{m\} \}$, then
the proposition holds from Step 1 or Step 2. 
Otherwise, by exchanging the indices if necessary, one may assume that there exists $1< k<m$ such that $\{k,\dots,m \} \in \mathcal{I}_{\rho}$, that is, 
$f=(f_1,\dots,f_m)^{\top} \in U_\rho$ satisfies $f_x \neq f_y$ for any $1 \le x \le k-1$ and $k \le y \le m$ and $f_k=f_{k+1}=\cdots =f_m$. 
Now we consider contractions of a function $f=(f_1,\dots,f_m)^\top \in \mathbb{R}^m$, given by
\[
\widetilde{f}:=\left( \begin{array}{c} f_1 \\ \vdots \\ f_{k-1} \\ \sum_{x=k}^n f_x\end{array} \right)  \in \mathbb{R}^k, \qquad 
\widehat{f}:=\left( \begin{array}{c} f_1 \\ \vdots \\ f_{k-1} \\ f_k \end{array} \right) \in \mathbb{R}^k, 
\]
and also consider a contraction of a matrix $A=(a_{ij})_{1\le ij \le m} \in M_m(\mathbb{R})$, given by
\[
\widetilde{A}:=\left( \begin{array}{cccc} 
a_{11} & \cdots & a_{1 k-1} & \sum_{j=k}^m a_{1j} \\
\vdots & \ddots & \vdots & \vdots \\
a_{k-1 1} & \cdots & a_{k-1 k-1} & \sum_{j=k}^m a_{k-1j} \\
\sum_{i=k}^m a_{i 1} & \cdots & \sum_{i=k}^m a_{i k-1} & \sum_{i,j=k}^m a_{ij}
\end{array} \right) \in M_k(\mathbb{R}). 
\]
One may also consider contractions of a function in $K(z)^m$ or in $\mathbb{R}(z)^m$, and a matrix in $M_m(K(z))$ or in $M_m(\mathbb{R}(z))$ 
in the same manner. 
Note that if $f \in U_{\rho}$, that is, $f_k=\cdots=f_m$, then $\widetilde{(A f)}=\widetilde{A} \widehat{f}$. 
Let $\widetilde{\rho} \in R_k$ be an index given by the relation $U_{\widetilde{\rho}}=\{\widehat{f} \in \mathbb{R}^k \;;\; f \in U_{\rho}\} \subset \mathbb{R}^k$. 

Then it is seen that there exists a contraction $\widetilde{H}=(\widetilde{V},\widetilde{E},\widetilde{w},\widetilde{d})$ of the hypergraph $H$ with $\widetilde{V}=\{1,\dots,k\}$ such that 
$\widetilde{G_{H,z}(f)}=G_{\widetilde{H},z}(\widehat{f})$ holds for any $f \in U_{\rho}$. 
Indeed, let us prepare the following notations:
\[
\begin{array}{l}
\pi : V=\{1,\dots,m\} \to \widetilde{V}:=\{1,\dots,k\}, \quad \pi(x):=\left\{ \begin{array}{cc} x & (x<k) \\ k & (x \ge k), \end{array} \right. \\
\pi : E \to \widetilde{E}:=\{\widetilde{e} \;;\; e \in E\}, \quad \pi(e)=\widetilde{e}:=\{\pi(x) \;;\; x \in e \} \subset \widetilde{V}, \\
\widetilde{w} : \widetilde{E} \to \mathbb{R}_{>0}, \quad \widetilde{w}_{\widetilde{e}}:=\sum_{e \in \pi^{-1}(\widetilde{e})} w_e, \\
\widetilde{d} : \widetilde{V} \to \mathbb{R}_{>0}, \quad \widetilde{d}_{x}:=\left\{ \begin{array}{cc} d_x & (x<k) \\ \sum_{y=k}^m d_y & (x \ge k). \end{array} \right. \\
\end{array}
\]
Since each matrix $S_{xy}$ with $x<y$ satisfies 
\[
\widetilde{S_{xy}}=\left\{ \begin{array}{cc} 
S_{xy} & (x<y <k) \\
S_{xk} & (x<k \le y) \\
0 & (k \le x<y), 
\end{array}\right.
\]
$\widetilde{\mathcal{M}_{e,\rho}}:=\{\widetilde{M} \;;\; M \in \mathcal{M}_{e,\rho}\}$ satisfies $\widetilde{\mathcal{M}_{e,\rho}}=\mathcal{M}_{\widetilde{e},\widetilde{\rho}}$ for any $e \in E$. 
As $\widetilde{D}=\mathrm{diag}(\widetilde{d}_1,\dots, \widetilde{d}_k)$, we have 
\[
\widetilde{G_{H,z}(f)}=(\widetilde{D}+z\sum_{e \in E} w_e \widetilde{\mathcal{M}_{e,\rho}})(\widehat{f})
=(\widetilde{D}+z\sum_{\widetilde{e} \in \widetilde{E}} \widetilde{w}_{\widetilde{e}} \mathcal{M}_{\widetilde{e},\widetilde{\rho}})(\widehat{f})
=G_{\widetilde{H},z}(\widehat{f})
\]
for any $f \in U_{\rho}$. 

By our assumption of the induction, there exists a symmetric matrix $N_{\widetilde{\rho},z} =(c_{ij}) \in M_k(K(z))$ 
such that $N_{\widetilde{\rho},\lambda}$ has non-negative entries for any $\lambda>0$ 
and $(N_{\widetilde{\rho},\lambda} \circ G_{\widetilde{H},\lambda})(\widehat{f})=\widehat{f}$ holds for any $f \in U_{\rho}$. 
Since $\widetilde{G_{H,\lambda}(f)}=\widetilde{G_{H,\lambda}}(\widehat{f})=G_{\widetilde{H},\lambda}(\widehat{f})$, 
we have $(N_{\rho,\lambda} \circ G_{H,\lambda})(f)=f$ for any $f \in U_{\rho}$, where 
\[
N_{\rho,z}:=\left( \begin{array}{ccccc} 
c_{11} & \cdots & c_{1 k} & \cdots & c_{1k} \\
\vdots & \ddots & \vdots & & \vdots \\
c_{k 1} & \cdots & c_{k k} & \cdots & c_{kk} \\
\vdots & & \vdots & & \vdots \\
c_{k 1} & \cdots & c_{k k} & \cdots & c_{kk}
\end{array} \right) \in M_m(K(z))
\]
is also a symmetric matrix with $N_{\rho,\lambda}$ having non-negative entires for $\lambda>0$. The proposition is established. 
\end{proof}

\subsection{Linear programming}

In order to prove Theorem \ref{thm:existenceofRicci}, we prepare the following lemma. 

\begin{lem} \label{lem:convconeofimage}
Let $H$ be a generalized hypergraph. 
For any $\rho \in R_n$, there exist vectors $v_1(z),\dots,v_k(z) \in K(z)^n$ such that for any $\lambda>0$, the closure of the image $G_{\lambda}(U_{\rho})$ is expressed as
\[
\overline{G_{\lambda}(U_{\rho})}=\mathrm{ConvCone}_{\mathbb{R}}(\{v_1(\lambda),\dots,v_k(\lambda)\}):= \sum_{i=1}^k \mathbb{R}_{\ge 0} \cdot v_i(\lambda), 
\]
where $G_{\lambda}$ is given in (\ref{eqn:multivaluedGfunction}). 
\end{lem}

\begin{proof}
We use the notations in the proof of Proposition \ref{prop:piecewiselinear}. For each element $f=(f_1,\dots,f_n)^{\top} \in U_{\rho}$, let $g_1^f> \cdots >g_l^f$ be given by $\{f_1,\dots,f_n\}=\{g_1^f,\dots,g_l^f\}$ 
and put $h_i^f:=g_i^f-g_{i+1}^f$ for $1 \le i \le l-1$ and $h_l^f:=g_{l}^f$, which gives a one-to-one correspondence between the sets $U_{\rho}$ and 
$\{(h_1,\dots,h_l)^{\top}  \;;\; h_1>0, \dots,h_{l-1}>0,h_l \in \mathbb{R}\}$. 
Since $f_i=g_k^f=\sum_{j=k}^l h_j^f$ for some $1 \le k \le l$, (\ref{eqn:Ldecomposition}) and (\ref{eqn:Leexpression}) show that 
\begin{align*}
G_zf= &Df+z \sum_{e\in E}w_e (f_e^+-f_e^-) \text{Conv}(\{ \delta_x - \delta_y \;;\; x \in e \cap I_{e,\rho}^+, \; y \in e \cap I_{e,\rho}^- \}) \notag \\
=& \sum_{j=1}^l h_j^f \left\{ \eta_j + z \sum_{J_e \ni j} w_e \text{Conv}(\{ \delta_x - \delta_y \;;\; x \in e \cap I_{e,\rho}^+,\; y \in e \cap I_{e,\rho}^- \}) \right\} 
\end{align*}
for $f\in U_{\rho}$, where $J_e=\{k^+,k^++1,\dots,k^--1\} \subset \{1,\dots,l-1\}$ is given by $f_e^{\pm}=g_{k^{\pm}}^f$, and 
$\eta_j=(\eta_{j1},\dots, \eta_{jn})^{\top} \in \mathbb{Q}(d)^n$ is given by
\[
\eta_{ji}=\left\{\begin{array}{cc}
d_i & (f_i \ge g_j^f) \\
0 & (f_i < g_j^f). 
\end{array} \right. 
\]
Note that the definitions of $J_e$ and $\eta_j$ are independent of the choice of $f\in U_{\rho}$. 
Thus $\overline{G_{\lambda}(U_{\rho})} \subset \mathbb{R}^n$ is the set of points 
\begin{align}
\sum_{j=1}^l h_j \left\{ \eta_j + \lambda \sum_{J_e \ni j} w_e \text{Conv}(\{ \delta_x - \delta_y \;;\; x \in e \cap I_{e,\rho}^+,\; y \in e \cap I_{e,\rho}^- \}) \right\}, \label{eqn:convconeexpression}
\end{align}
where $(h_1,\dots,h_l)^{\top}$ varies over $\{h_1\ge0, \dots,h_{l-1}\ge 0,h_l \in \mathbb{R}\}$. 
Moreover we also notice that 
\begin{itemize}
\item the multiplication of a convex set $\mathrm{Conv}(\{u_i\})$ by a scalar $c \in \mathbb{R}$ is also a convex set: 
$c \cdot \mathrm{Conv}(\{u_i\})=\mathrm{Conv}(\{c \cdot u_i\})$, 
\item the Minkowski sum of convex sets $\mathrm{Conv}(\{u_i\})$ and $\mathrm{Conv}(\{v_j\})$ is also a convex set:  
$\mathrm{Conv}(\{u_i\})+\mathrm{Conv}(\{v_j\})=\mathrm{Conv}(\{u_i+v_j\})$, 
\item the multiplication of a convex set $\mathrm{Conv}(\{u_i\})$ by the (non-negative) real numbers $\mathbb{R}$ ($\mathbb{R}_{\ge0}$) is a convex cone:  
$\mathbb{R} \mathrm{Conv}(\{u_i\})=\mathrm{ConvCone}(\{u_i\} \cup \{-u_i\})$ $\left( \mathbb{R}_{\ge0} \mathrm{Conv}(\{u_i\})=\mathrm{ConvCone}(\{u_i\})\right)$, and 
\item the Minkowski sum of convex cones $\mathrm{ConvCone}(\{u_i\})$ and $\mathrm{ConvCone}(\{v_j\})$ is a convex cone:  
$\mathrm{ConvCone}(\{u_i\})+\mathrm{ConvCone}(\{v_j\})=\mathrm{Conv}(\{u_i\} \cup \{v_j\})$. 
\end{itemize}
Hence the proposition follows from the expression (\ref{eqn:convconeexpression}). 
\end{proof}

From now on, we assume that $H$ is a hypergraph. 
As $G_{\lambda}f=(I+\lambda\calL) \circ D(f)$, one has 
\begin{align} \label{eqn:relationJN}
N_{\rho,\lambda}(g) =D^{-1} \circ J_{\lambda}(g)
\end{align}
for $g \in G_{\lambda}(U_{\rho})$. 
Since $J_\lambda$ is a (single-valued) continuous function by Lemma \ref{lem:conti-of-resolvent}, the relation (\ref{eqn:relationJN}) holds for 
$g \in \overline{G_{\lambda}(U_{\rho})}$. 
Here it follows from (the proof of) Proposition \ref{prop:finitenorm} that the $\lambda$-nonlinear Kantorovich difference $\KD_{\lambda}(x_0,y_0)$ of $x_0 \in V$ and $y_0 \in V$ is given by
   \begin{align}
    \KD_{\lambda}(x_0,y_0)=\sup\left\{\la J_{\lambda}g,\delta_{x_0}-\delta_{y_0}\ra\;;\;g\in F \right\}, \notag
   \end{align} 
where $F$ is the set of functions $g\in\mathbb{R}^n$ satisfying the conditions
\[
\text{(a) } 0 \le g_x \leq d_x \cdot \diam(H) \quad (x\in V), \quad \text{(b) } \la g,\delta_x-\delta_y\ra\leq d(x,y) \quad (x,y\in V). 
\]
Since $N_{\rho, \lambda}$ and $D$ are symmetric matrices and $\{G_{\lambda}(U_{\rho})\}_{\rho \in R}$ covers the whole space $\mathbb{R}^n$, we have 
   \begin{align}
    \KD_{\lambda}(x_0,y_0)= & \max_{\rho \in R_n} \,\sup\left\{\la J_{\lambda}g,\delta_{x_0}-\delta_{y_0}\ra\;;\;g\in F \cap \overline{G_{\lambda}(U_{\rho})} \right\} \label{eqn:maxsupKD} \\
    = & \max_{\rho \in R_n} \,\sup\left\{\la g,D \circ N_{\rho,\lambda}(\delta_{x_0}-\delta_{y_0})\ra\;;\;g\in F \cap \overline{G_{\lambda}(U_{\rho})} \right\}. \notag
   \end{align} 

Now we consider an order $<$ on $K(z)$, given so that 
$p(z), q(z) \in K(z)$ satisfy $p(z)< q(z)$ if and only if $q(z)-p(z)=z^k r(z)$ for 
some $k \in \mathbb{Z}$ and $r(z) \in K(z)$ with $0<r(0)<\infty$. 
In other words, $p(z), q(z) \in K(z)$ satisfy $p(z)< q(z)$ if and only if there exists $\lambda_0>0$ such that 
$p(\lambda)<q(\lambda)$ for any $\lambda \in \mathbb{R}$ with $0<\lambda< \lambda_0$. 
Then $<$ becomes a total order on $K(z)$ and thus $(K(z),<)$ is an ordered field. 

With the notation in Lemma \ref{lem:convconeofimage}, we consider the convex cone 
\[
W_\rho:=\mathrm{ConvCone}_{K(z)}(\{v_1(z),\dots,v_k(z)\})=\sum_{i=1}^k K(z)_{\ge 0} \cdot v_i(z) \subset K(z)^n. 
\]
It should be noted that some of the concepts of linear programming over the real numbers, such as Farkas-Minkowski-Weyl theorem and the simplex method, can be easily 
extended to that over an arbitrary ordered field (see \cite{jeroslow1973asymptotic, joswig2015linear}). Farkas-Minkowski-Weyl theorem says that there exist vectors $w_1(z),\dots, w_m(z) \in K(z)^n$ such that 
the convex cone $W_{\rho}$ is expressed as
\[
W_{\rho}=\{ g(z) \in K(z)^n \;;\, \langle g(z),w_i(z) \rangle \le 0 \; (1\le i \le m) \}. 
\]
In viewing (\ref{eqn:maxsupKD}), we consider the linear program $\mathrm{LP}(z)$: 
\begin{align}
\text{maximize} \quad & \la g(z),D \circ N_{\rho,z}(\delta_{x_0}-\delta_{y_0})\ra \notag \\
\text{subject to} \quad & \text{(a) } \, 0 \le g(z)_x \leq d_x \cdot \diam(H) \quad (x\in V) \notag \\
 & \text{(b) } \, \la g(z),\delta_x-\delta_y\ra\leq d(x,y) \quad (x,y\in V) \notag \\
 & \text{(c) } \, \langle g(z),w_i(z) \rangle \le 0 \quad  (1\le i \le m). \notag
\end{align}
As the range of $g(z)$ is bounded, the simplex method guarantees that there exists an optimal solution $g^{(\rho)}(z) \in K(z)^n$ to the linear program $\mathrm{LP}(z)$ with optimal value $h^{(\rho)}(z) \in K(z)$. 
Moreover the following proposition holds (see \cite{jeroslow1973asymptotic}, 2.3, \cite[Corollary 2]{joswig2015linear}). 
\begin{prop} \label{prop:linearprogram}
Under the above notations, there exists $\lambda_{\rho} \in \mathbb{R}_{>0}$ such that for every $0<\lambda < \lambda_{\rho}$, 
$g^{(\rho)}(\lambda) \in K(\lambda)^n$ is an optimal solution to the linear program $\mathrm{LP}(\lambda)$ with optimal value $h^{(\rho)}(\lambda) \in K(\lambda)$. 
\end{prop}

\begin{proof}[Proof of Theorem \ref{thm:existenceofRicci}]
As $\#R_n<\infty$, $h^*(z):=\max\{h^{(\rho)}(z) \;;\; \rho \in R_n\}$ and $\lambda_*:=\min\{\lambda_{\rho} \;;\; \rho \in R_n\}$ satisfy $h^*(z) \in K(z)$ and $\lambda_*>0$.   
Thanks to (\ref{eqn:maxsupKD}) and Proposition \ref{prop:linearprogram}, the Kantorovich difference $\KD_{\lambda}(x_0,y_0)$ is expressed as 
$\KD_{\lambda}(x_0,y_0)=h^*(\lambda)$ for any $0<\lambda < \lambda_*$. 
Since $h^{*}(z)$ is a rational function of $z$, the limit in (\ref{eqn:existenceofRicci}) exists, which establishes the theorem. 
\end{proof}

\section{More general settings}\label{sec:general-settings}


Our arguments so far are applicable to more general settings for 
 submodular transformations. Here, submodular transformation  
is a vector valued set function consisting of submodular functions.   
In this section, we review about submodular functions, submodular 
transformations, and these Laplacian and show some examples. 
We also give a sufficient condition for a submodular transformation 
to be able to straightforwardly generalize the curvature notions in Section~\ref{sec:def-of-curvature-hypergraphs} and theorems in Section~\ref{sec:applications}.  
For more details about submodular transformations, see~\cite{yoshida2019cheeger}. 

\subsection{Submodular function} 
\label{subsec:submodular-functions}

Let $V$ be a nonempty finite set. 
A function 
$F\colon 2^V \to \RR$ is 
a submodular function if for any $S, T \subset V$, $F$ satisfies 
\[
 F(S) + F(T) \geq F(S\cup T) + F(S\cap T). 
\]
An element $v \in V$ is relevant in $F\colon 2^V \to \RR$ 
if there is a $S \subset V$ such that $F(S) \neq F(S\cup\{ v\})$. 
We say that $v$ is irrelevant in $F$ if $v$ is not relevant in $F$. 
We define the support $\supp(F)$ of $F$ as the set of elements which are 
relevant in $F$. 
A set function $F\colon 2^V \to \RR$ is symmetric if $F(S) = F(V{\setminus}S)$ holds for any $S$. 
We say that $F$ is normalized if $F(V) = 0$. 

\begin{example}
 Let $H = (V, E)$ be a hypergraph, and $e \in E$ a hyperedge. 
Then, the cut function $F_e$ of $e$ defined as follows is a submodular function: 
\[
 F_e(S) = \begin{cases}
 	1 & \text{if } e\cap S \neq \emptyset \text{ and } e\cap (V{\setminus}S) \neq \emptyset, \\ 
 	0 & \text{otherwise}.
 \end{cases}
\] 
It is easy to show that a vertex $v\in V$ is relevant in $F_e$ if and only if $v \in e$. Furthermore, $F_e$ is symmetric and normalized.  
\end{example}

For a submodular function $F\colon 2^V \to \RR$, we define 
\begin{align*}
	P(F) &:= \left\{ g \in \RR^V\ ; \sum_{x \in S} g(x) \leq F(S) \text{ for any } S \subset V \right\} \text{ and } \\ 
	B(F) &:= \left\{ g \in P(F)\ ; \sum_{x \in V} g(x) = F(V) \right\}
\end{align*}
called the submodular polyhedron and the base polytope respectively. 
Then, it is known that $B(F)$ is a bounded polytope.  

The Lov\'{a}sz extension $f \colon \RR^V \to \RR$ of a submodular 
 function 
$F\colon 2^V \to \RR$ is defined by 
\[
 f(g) := \max_{\ttb\in B(F)} \ttb^\top g. 
\]
It is known that $f(\chi_S) = F(S)$ for any $S\subset V$. Here, 
$\chi_S$ is the characteristic function of $S$. In particular, $f$ 
is indeed an extension of $F$. It is also known that 
the Lov\'{a}sz extension $f$ of a submodular function $F$ is convex (\cite[Proposition3.6]{FrancisBach2013}). 

For the Lov\'{a}sz extension $f$ of a submodular function $F$, 
we set 
\[
 \partial f(g) := \argmax_{\ttb \in B(F)} \ttb^\top g. 
\]
Then, it is known that 
$\partial f(g)$ is the sub-differential of $f$ at $g$. 

\subsection{Submodular transformation and submodular Laplacian}
\label{subsec:submodular-transformations}

Let $V$ and $E$ be nonempty finite sets. 
A function 
$F\colon 2^V \to \RR^E; S\mapsto F(S) = (F_e(S))_{e\in E}$ is called a submodular transformation if each $F_e$ is a submodular function. 
A submodular transformation $F$ is symmetric (resp. normalized) if any $F_e$ is symmetric (resp. normalized). 

The Lov\'{a}sz extension 
$f\colon \RR^V \to \RR^E$ of a submodular 
transformation $F$ is defined by $f = (f_e)$ such that $f_e$ is the 
Lov\'{a}sz extension of $F_e$. 

For a submodular transformation $F\colon \RR^V \to \RR^E$, 
we consider a weight function $w \colon E \to \RR_{>0}$. 
Then, we call the quadruple $(V,E,F,w)$ a weighted submodular transformation. We stand for the quadruple as $F$.   
We define the degree $d_x$ for $x \in V$ by 
$d_x := \sum_{e \in E; x \in \supp(F_e)} w_e$ and the volume 
$\vol(S)$ of $S \subset V$ by $\vol(S) := \sum_{x\in S} d_x$.  
For $x, y\in V$, $x$ and $y$ is adjacent, denoted by $x\sim y$, 
if there exists an element $e\in E$ such that $x,y\in \supp(F_e)$. 
By this relation, we can define the distance function $d\colon V\times V \to \RR_{\geq 0}$ and connectivity of $F$ as 
in Section~\ref{subsec:prelim-hypergraphs}.  

We define the degree matrix 
$D := \mathrm{diag}(d_x)_{x\in V} \in \RR^{V\times V}$. We remark that
 if $F$ is connected, $D$ is invertible. 


Let $F=(V,E,F,w)$ be a submodular transformation. 
Then, we define the submodular Laplacian $L \colon \RR^V \to 2^{\RR^V}$ by 
\[
 L(g) := \left\{ \sum_{e\in E} w(e) \ttb_e \ttb_e^\top g ; \
 \ttb_e \in \partial f_e(g)
 \right\} \subset \RR^V.  
\]
We call $\ca{L}:=L\circ D^{-1}$ the normalized Laplacian. 
We set the inner product $\la f, g \ra:= f^\top D^{-1} g$ and 
consider $(\RR^V, \la\cdot,\cdot \ra)$ as a Hilbert space. 
Then, by a similar argument as 
in~\cite[Lemma~14, Lemma~15]{ikeda2018finding}, the following holds: 
\begin{prop}\label{prop:maximal-monotone-submodular}
The normalized Laplacian $\ca{L}$ is a 
maximal monotone operator on the Hilbert space $(\RR^V, \la\cdot,\cdot\ra)$. 
\end{prop}

More strongly, 
the normalized Laplacian $\ca{L}$ is the sub-differential 
of the convex function $Q\colon \RR^V \to \RR$ defined by 
\[
 Q(\ol{g}) = \frac{1}{2}\sum_{e\in E} w_e f_e(\ol{g})^2, 
\]
where $\ol{g} = D^{-1}g$ with $g\in \mathbb{R}^V$.

By Proposition~\ref{prop:maximal-monotone-submodular}, 
 we can define the resolvent $J_\lambda$, the canonical restriction 
 $\ca{L}_0$, and the heat semigroup $\{h_t\}_{t\ge 0}$ for the Laplacian $\ca{L}$. 
Then, the straight extension of Lemma~\ref{lem:homogenity-of-Laplacian} holds. 

We define $\pi \in \RR^V$ as $\pi(x) = d_x/\vol(V)$. Then, the following holds: 
\begin{lem}[{\cite[Lemma~3.1]{yoshida2019cheeger}}]\label{lem:vanish-Laplacian-submodular}
 We assume that $F$ is normalized, i.e., $F_e(V) = 0$ for any $e\in E$. Then, $\ca{L}(\pi) = 0$ holds. 
\end{lem}

By Lemma~\ref{lem:vanish-Laplacian-submodular}, the similar lemmas as 
Lemma~\ref{lem:homogenity-of-Laplacian} and Lemma~\ref{lem:homogenity-of-resolvent} hold for the normalized submodular Laplacian $\ca{L}$. 
This implies that by similar arguments, we can obtain the 
straightforward extensions of 
definitions and theorems in 
Section~\ref{sec:def-of-curvature-hypergraphs} and 
Section~\ref{sec:applications} for any normalized submodular 
transformation $F$ with the normalized submodular Laplacian 
$\ca{L}$ for $F$. 

\subsection{Examples}\label{subsec:submodular-examples}

In \cite{yoshida2019cheeger}, Yoshida gave many examples of submodular transformations such as 
undirected graphs (Example~1.1, 1.2, and 1.4), directed graphs (Example~1.5), hypergraphs (Example~1.6), submodular hypergraphs 
(Example~1.7), mutual information (Example~1.8), and directed information (Example~1.9). 
We here give another example: 
\begin{example}[directed hypergraph] 
A weighted directed hypergraph $H$ is defined as 
the triple $H = (V, E, w)$ of a set of vertices $V$, a set of hyperarcs $E \subset 2^V \times 2^V$, and a weight function $w\colon E \to \RR_{>0}$. 
Here, a hyperarc $e \in E$ is an ordered pair $(t_e, h_e)$ of 
a set of tails $t_e$ and a set of heads $h_e$. 
If the identities $|t_e|=|h_e|=1$ hold for any $e \in E$, then $H$ is a usual 
directed graph. If $t_e = h_e$ holds for any $e\in E$, 
then $H$ can be regarded as an undirected hypergraph\footnote{This specialization seems to be strange. However, from the viewpoint of submodular transformation, this looks natural. Indeed, under the assumption $t_e = h_e$, 
the cut function is same as that of undirected hypergraphs}. 
Hence, a directed hypergraph is a generalization of directed graphs 
and hypergraphs. 

We define the set function $F_e\colon 2^V \to \RR$ as the 
cut function for $e = (t_e, h_e)$, i.e., 
\[
 F_e(S) := \begin{cases}
 	1 & \text{if } S\cap t_e \neq \emptyset \text{ and } 
 	(V{\setminus}S)\cap h_e \neq \emptyset, \\ 
 	0 & \text{otherwise.}  
 \end{cases}
\]
Then, it is easy to show that $F_e$ is submodular. 
Hence, the quadruple $F = (V, E, F=(F_e)_{e}, w)$ becomes 
a submodular transformation.  
We remark that $F$ is normalized and not symmetric. 

For this submodular transformation, by a simple calculation from definition of $B(F_e)$, we have 
\begin{align}
 B(F_e) = 
 \mathrm{Conv}(\{ \delta_x - \delta_y; x \in t_e, y \in h_e \}\cup 
 \{ 0 \}). \label{eq:base-polytope-directed-hyper} 
\end{align}
The base polytope for hypergraph~\eqref{eq:base-polytope-hyper} is 
a realization of this for $t_e = h_e$. 
By the representation~\eqref{eq:base-polytope-directed-hyper}, 
the Lov\'{a}sz extension $f_e$ of $F_e$ is written as 
\[
 f_e(g) = \max\{ \max\{ g(x) - g(y) ; x \in t_e, y\in h_e\}, 0\}. 
\]
\end{example}

We note that for all examples introduced in this subsection, $F$ is 
 normalized, i.e., $F(V) = 0$.
Hence, the similar definitions of coarse Ricci curvatures for 
$F$ as in Section~\ref{sec:def-of-curvature-hypergraphs} and the similar theorems as in Section~\ref{sec:applications} hold. 

\section{Concluding Remark}
Comparing properties of curvatures for the examples in this paper with those of other curvatures introduced by \cite{ozawa2020geometric, eidi2020ollivier} is an interesting problem.  The definitions of Ricci curvature in these two papers deeply related to random walks. For the authors, the canonical random walks on hypergraphs are not clear. Of course, one can define the random walk as in \cite{eidi2020ollivier}, which seems related to the clique expansion. It is unclear for the authors that the clique expansion of a hypergraph reflect the characteristics of its own hypergraph structure. At least, because the hypergraph Laplacian is multivalued and nonlinear, there was no canonical way to define the transition probabilities of random walkers using it. For these reasons, still we do not know any essential relation between theirs and ours. We leave it for a future work.


\vspace{1cm}
No data associate for the submission

\bibliography{main.bib}

 




 \end{document}